\documentclass[12pt]{article}
\usepackage{amssymb,amsmath,amsthm,enumerate}
\usepackage{psfrag}
\usepackage{dcolumn}
\usepackage{color}
\usepackage{bm}
\usepackage{epsfig}
\usepackage{graphicx,subfigure}
\usepackage{psfrag}
\usepackage{graphics}
\usepackage{color}

\newtheorem{remark}{Remark}
\newtheorem{theorem}{Theorem}
\newtheorem{proposition}{Proposition}

\newtheorem{definition}{Definition}

\begin{document} 

\title{Degenerate anisotropic elliptic problems \\
and magnetized plasma simulations.}
\author{St\'ephane Brull$^{(1)}$, Pierre Degond$^{(2,3)}$, Fabrice Deluzet$^{(2,3)}$ }
\date{}
\maketitle

\begin{center}
1-Institut de Math\'{e}matiques de Bordeaux UMR 5251\\
Equipe Math\'{e}matiques Appliqu\'{e}es de Bordeaux (MAB)\\
Universit\'{e} Bordeaux 1\\
351, cours de la Lib\'{e}ration - 33405 TALENCE cedex
FRANCE\\
email: Stephane.Brull@math.u-bordeaux1.fr
\end{center}

\begin{center}
2-Universit\'{e} de Toulouse; UPS, INSA, UT1, UTM ;\\ 
Institut de Math\'{e}matiques de Toulouse ; \\
F-31062 Toulouse, France. \\
email: fabrice.deluzet@math.univ-toulouse.fr \\
\end{center}

\begin{center}
3-CNRS; Institut de Math\'{e}matiques de Toulouse UMR 5219 ;\\ 
F-31062 Toulouse, France.\\
email: pierre.degond@math.univ-toulouse.fr \\
\end{center}

\vspace{0.5 cm}

\begin{abstract}
This paper is devoted to the numerical approximation of a degenerate  anisotropic
elliptic problem. The numerical method is designed for arbitrary space-dependent anisotropy directions and does not require any specially adapted coordinate system. It is also designed to be equally accurate in the strongly and the mildly anisotropic cases. The method is applied to the Euler-Lorentz system, 
in the drift-fluid limit. This system provides a model for magnetized plasmas.
\end{abstract}



\medskip
{\bf Acknowledgements:} This work has been partially
supported by the Marie Curie Actions of the European Commission in the
frame of the DEASE project (MEST-CT-2005-021122), by the 
'F\'ed\'eration de recherche CNRS sur la fusion par confinement magn\'etique', 
by the Association Euratom-CEA in the framework of the contract 
'Gyro-AP' (contract \# V3629.001 avenant 1) and by the University Paul Sabatier
in the frame of the contract 'MOSITER'.
This work was performed while the first author held a post-doctoral position funded by 
the Fondation 'Sciences et Technologies pour 
l'A\'eronautique et l'Espace', in the frame of the project 'Plasmax' 
(contract \# RTRA-STAE/2007/PF/002). The authors wish to express their gratitude
to G.~Falchetto, X.~Garbet and M.~Ottaviani from CEA-Cadarache and G. Gallice and C. Tessieras from CEA-Cesta for
fruitful discussions and encouragements.


\setcounter{equation}{0}
\section{Introduction}
\label{DA1}

This paper discusses the numerical resolution of degenerate anisotropic elliptic problems of the form:  
\begin{eqnarray}
\label{P1} - \left(   b \cdot { \nabla} \right) \left( {  \nabla }
\cdot ( b \, \phi^{\varepsilon}) \right) + \varepsilon \phi^{\varepsilon}
&=&f ^{\varepsilon} , \hspace*{2 mm} \mbox{in} \hspace*{2 mm} \Omega , 
\\ \label{bc1} \left(  b \cdot \nu \right) { \nabla } \cdot
\left(  b \, \phi^{\varepsilon } \right)&= &0 
\hspace*{2 mm} \mbox{on} \hspace*{2 mm} \partial \Omega ,
\end{eqnarray}
where $\Omega \subset \mathbb{R}^2$ or $\mathbb{R}^3$,
$f^\varepsilon$ is a given function, $b$ is a normalized vector field defining the anisotropy direction and $\varepsilon$ measures the strength of this
    anisotropy. In this expression $\nabla$  and $\nabla \cdot$ are
    respectively  the gradient and divergence operators.
 The unit outward normal at $x \in \partial \Omega$ is denoted by $\nu$.
 In the context of plasmas, $\varepsilon$ is related to the
 gyro period (i.e. the period of the gyration motion of the particles about the magnetic field lines), and the anisotropy direction  $b$ satisfies $b = B/|B|$ with the magnetic
 field $B$ verifying $\nabla \cdot B=0$. Eq. \eqref{P1} may also arise in other contexts, such as rapidly rotating flows, shell theory and may also be found when special types of semi-implicit time
  discretization of diffusion equations are used. 
  
  The elliptic
  equation is not in the usual divergence form due to an exchange
  between the gradient and divergence operators. However, the
  methodology would apply equally well to the operator $\nabla \cdot (
  (b \otimes b) \cdot \nabla   \phi))$, up to some simple changes. The 
  expression considered here is motivated by the application to the
  Euler-Lorentz system of plasmas. This application has already been considered in a previous study \cite{DDSV_JCP09} but we introduce two important developments. First the
present numerical method does not request the development of a special coordinate system adapted to $b$. In \cite{DDSV_JCP09}, $b$ was
assumed aligned with one coordinate direction.
Second, the present paper considers Neumann boundary conditions instead of Dirichlet ones as in \cite{DDSV_JCP09}. Although seemingly innocuous, this change brings in a considerable difficulty, linked with the degeneracy of the limit problem, as explained below.

A classical
discretization of problem (\ref{P1}), (\ref{bc1}) leads to an
 ill-conditioned linear system as $\varepsilon \to 0$. Indeed 
 setting formally $\varepsilon = 0$ in (\ref{P1}), (\ref{bc1}), we get:
 \begin{eqnarray}
   & & - ( b \cdot \nabla) \nabla \cdot (b\, \psi) = f^{(0)}\,, \quad \text{in } \Omega , \label{eq:Singular:limit_1}
\\
& &     (b \cdot \nu) \nabla \cdot (b\, \psi) = 0 \,, \quad  \text{on } \partial \Omega ,\label{eq:Singular:limit_2}
 \end{eqnarray}
with $f^{(0)} = \lim_{\varepsilon \to 0} f^\varepsilon$. 
 The homogeneous system associated to \eqref{eq:Singular:limit_1}, \eqref{eq:Singular:limit_2} admits an infinite number of solutions, namely  all
 functions $\psi$ satisfying $\nabla \cdot ( b \psi) = 0$. This
 degeneracy results from the Neumann boundary conditions \eqref{eq:Singular:limit_2} and would also occur if  periodic boundary conditions were used. On the other hand, \eqref{eq:Singular:limit_1} is not degenerate if supplemented with  Dirichlet
 or Robin conditions, which was the case considered in \cite{DDSV_JCP09}. A standard numerical approximation of \eqref{eq:Singular:limit_1}, \eqref{eq:Singular:limit_2}  generates a matrix whose condition number blows up as $\varepsilon \to 0$, leading to very time consuming and/or poorly accurate solution algorithms. 
 
 To bypass these limitations, we
 follow the idea introduced in \cite{DDN_MMS10} and use a decomposition of
 the solution in its average along the $b$-field lines and a fluctuation about this average.
 This decomposition ensures an accurate computation
 of the solution for all values of $\varepsilon$. 
 In \cite{DDN_MMS10}, this decomposition approach was developed for a uniform $b$ and a coordinate system with one coordinate direction aligned with $b$. To extend this approach to arbitrary anisotropy fields $b$, a possible way is to use an adapted curvilinear coordinate system with one coordinate curve tangent to $b$. 
  This is the route followed by \cite{BBNY_prep}, which proposes an extension of
  \cite{DDN_MMS10} in the context of ionospheric plasma physics, where the 
  anisotropy direction is known analytically (given by the earth dipolar magnetic field). The approach developed
  here is different and aims at a method which does not request the generation of special curvilinear coordinates. Indeed, in the general case, computing such coordinates can be complex and costly, especially for time-dependent problems where $b$ evolves in time.

For this purpose, we solve a variational problem for each of the terms of the decomposition.
The main difficulty lies in the discretization of the functional spaces in
which each component of the solution is searched. 
In the present paper, this difficulty is solved by introducing two kinds of variational systems, one corresponding to a second-order elliptic problem (for the average) and one, to a fourth order system (for the fluctuation). An alternative to this method is proposed in \cite{DDLNN_sub}. It avoids the resolution of a fourth-order problem at the price of the introdution of Lagrange multipliers which lead to a larger system. In the present paper, we design a method which breaks the complexity of the problem in smaller pieces and requires less computer ressources. 

As an application of the method and a motivation for studying problem \eqref{eq:Singular:limit_1}, \eqref{eq:Singular:limit_2}, the drift-fluid limit of the isothermal Euler-Lorentz
  system is considered. These equations model the evolution of a magnetized plasma.  In this case, the anisotropy direction is that of the magnetic field and the parameter $\varepsilon$ is the reciprocal of
  the non dimensional cyclotron frequency. The drift-fluid limit $\varepsilon \to 0$ of the Euler-Lorentz system is singular because the momentum
  equation becomes degenerate. In this paper, we propose a scheme able to handle both the  
$\varepsilon \sim 1$ and $\varepsilon \ll 1$ regimes, giving rise
to consistent approximations of both the Euler-Lorentz model and its drift-fluid
limit, without  any
constraint on the space and time steps related to the possible small value of
$\varepsilon$. Schemes having such properties are referred to as \textit{Asymptotic-Preserving
(AP)} schemes. These schemes are particularly efficient in situations in which part of the simulation domain is in the
asymptotic regime and part of it is not. 
Indeed, in most practical cases, the parameter $\varepsilon$ assumes a local value which may change from one location to the next or which may evolve with time. 

The usual approach for
dealing with such occurences is through domain
decomposition: the full Euler-Lorentz model is used in the region
where $\varepsilon =O(1)$ and the drift-fluid limit model is used
where $\varepsilon \ll 1$. There are several drawbacks in using this approach. The
first one is the choice of the position of the interface (or
cross-talk region), which can influence the outcome of the
simulation. If the interface evolves in time, an algorithm for
interface motion has to be devised and some remeshing must be used to
ensure compatibility between the mesh and the interface, which requires heavy code developments and can be
quite CPU time consuming. Determining the right coupling strategy
between the two models can also be quite challenging and the outcome
of the numerical simulations may also depend on this choice. Because
these questions do not have straigthforward answers,
domain decomposition strategies often lack robustness and
reliability. Here, using the original model with an AP discretization
method everywhere prevents from these 
artefacts and permits to use the same code
everywhere for both regimes.  

We conclude this introductory section by some bibliographical remarks. 
In magnetized plasma simulations, many works are based on the use of curvilinear coordinate systems where one of the coordinate curves is tangent to the magnetic field (see e.g. \cite{Miyamoto_CH07}, the gyro-kinetic and gyro-fluid developments \cite{Beer_Hammett_96,  Dorland_Hammett_PhysFluidsB93, Grandgirard_CNSNS08, Hammett_Dorland_PPCF93} and the many attempts for generating specialized coordinate systems \cite{Beer_Cowley_Hammett_95, Boozer_Fluids82, Dhaeseleer_SSCP91, Dimits_PhysrevE93, Hamada_NuclFus62, Kaveeva_TPL04, Stern_JGR67}). The present work, together with \cite{DDLNN_sub} is one of the very few attempts to design numerical methods free of the use of special coordinate systems (see also \cite{Ottaviani_sub}). The key idea behind this method is the concept of Asymptotic Preserving (AP) schemes as described above. AP-schemes have first been introduced by S. Jin \cite{Jin_SISC99} in the context of diffusive limits of transport models. They have recently found numerous applications to plasma physics in relation e.g. to quasineutrality \cite{BCDS_JSC09, CDV_JCP07, DDNSV_JCP07, DLSV_EPB, DLV_SINUM08} and strong magnetic fields 
\cite{DDLNN_sub, DDN_MMS10, DDSV_JCP09} as well as to fluid-mechanical problems such as the small Mach-number limit of compressible fluids \cite{Deg_Tan_CICP_ta}. Other applications of AP-schemes can be found in \cite{Bue_Cor_M2AN02, Buet_Despres_JCP06, Car_Gou_Laf_JCP08, Fil_Jin_JCP10, Klar_SINUM99, Lem_Mie_SISC08, McC_Low_JCP08}. Numerical methods for anisotropic problems have been extensively
studied in the literature using numerous techniques such as domain decomposition techniques 
\cite{Giraud,Koronskij}, Multigrid methods, smoothers 
\cite{ICASE}, the $hp$-finite element method \cite{Melenek}. However, these methods are based on a discretization of the anisotropic PDE as it is written. The method presented here as well as in \cite{DDSV_JCP09, DDN_MMS10, DDLNN_sub} relies on a totally different concept, namely viewing the anisotropy as a singular perturbation and using Asymptotic-Preserving techniques. 
  
\medskip 
This paper is organized as follows. In section~\ref{elip} the
  solution methodology for the degenerate anisotropic elliptic problem
  (\ref{P1}), (\ref{bc1}) is
  detailed. Section~\ref{num} is devoted to the
  discretization strategy. In section~\ref{DA2}  the drift-fluid limit of the isothermal
  Euler-Lorentz system is introduced. The AP-scheme is
  derived, giving rise to the anisotropic elliptic problem (\ref{P1}), (\ref{bc1}).
  The numerical method for the anisotropic elliptic
  problem is validated in
  section~\ref{simu:ellip}. Finally a numerical application to the
  Euler-Lorentz system is given in section~\ref{sec:comp}.


\setcounter{equation}{0}
\section{A decomposition method for degenerate ani\-sotropic elliptic  problems} 
\label{elip}

We first present the methodology in the simpler case of a uniform $b$-field. The method will then be extended to an arbitrary $b$-field.

\subsection{Overview of the method in the uniform $b$-field case}
\label{sec:overview:AP}

A two dimensional configuration is considered in this section, with the position variable $(x,y)$ belonging to a square domain $(x,y) \in \Omega=[0,1]\times[0,1] \subset {\mathbb R}^2$. The $b$ field is assumed uniform, equal to the unit vector pointing in the  $y$ direction. In this case, the singular
perturbation problem (\ref{P1}), (\ref{bc1}) reads: 
\begin{eqnarray} 
\label{eq1}
\varepsilon \phi^{\varepsilon}(x,y) - \frac{\partial^2}{\partial y^2} \phi^{\varepsilon}(x,y) &=&
f^{\varepsilon}(x,y), \hspace*{2 mm} \mbox{in} \hspace*{2 mm} ]0,1[ \times ]0,1[,
 \\ 
 \label{eqb} 
 \frac{\partial}{\partial y} \phi^{\varepsilon}(x,y) &=& 0, \hspace*{2
   mm} \mbox{for} \hspace*{1  mm}  y=0 \hspace*{1  mm} \mbox{or} \hspace*{1  mm} y=1 . 
\end{eqnarray}
We assume that:
\begin{equation}
\lim_{\varepsilon \to 0} \left( \frac{1}{\varepsilon} \int_0^1 f^\varepsilon (x,y) \, dy \right) \quad  \mbox{exists and is finite,} \quad \forall x \in [0,1]. 
\label{eq:hypo_cancel}
\end{equation}

This framework is similar to \cite{DDN_MMS10}. Here, we recall the bases of the methodology. The problem
is well posed for all $\varepsilon > 0$ but a standard discretization may lead to ill-conditionned matrices when $\varepsilon \ll 1$. Indeed if $\varepsilon$
  is formally set to zero, we get the following
  degenerate problem
  \begin{equation}\label{eq:simplified:degenerated}
    \begin{split}
       - \frac{\partial^2}{\partial y^2} \psi(x,y) &=
      f^{(0)}(x,y), \hspace*{2 mm} \mbox{in} \hspace*{2 mm} ]0,1[ \times ]0,1[,
      \\ \frac{\partial}{\partial y} \psi(x,y) &= 0, \hspace*{2
        mm} \mbox{for} \hspace*{1  mm}  y=0 \hspace*{1  mm} \mbox{or} \hspace*{1  mm} y=1 , 
    \end{split}
  \end{equation}
  assuming that $f^\varepsilon$ has the following expansion $f^\varepsilon= f^{(0)} + \varepsilon f^{(1)} + o(\varepsilon)$. 
  This system admits a solution under the compatibility condition
  $\int_0^1 \, f^{(0)} (x,y) \,dy= 0$ for all $x \in [0,1]$, which is satisfied thanks to hypothesis (\ref{eq:hypo_cancel}).  However the solution is not unique. Indeed, if
  $\psi$ verifies \eqref{eq:simplified:degenerated} then $\psi+\zeta$
  is also a solution for all functions $\zeta=\zeta(x)$ which depend on the
  $x$-coordinate only. 

On the other hand, the limit $\phi^{(0)} = \lim_{\varepsilon \to 0} \phi^\varepsilon$ is unique. Indeed, it is easy to see that the solution $\tilde \psi$ of (\ref{eq:simplified:degenerated}) such that $\int_0^1 \, \tilde \psi (x,y) \,dy= 0$ for all $x \in [0,1]$ is unique. Since $\phi^{(0)}$ is a particular solution of (\ref{eq:simplified:degenerated}), it can be written 
\begin{equation}
\phi^{(0)} = \tilde \psi + \zeta(x).
\label{eq:phi=psi+zeta}
\end{equation}
 In order to determine $\zeta$ , we integrate \eqref{eq1} with respect to $y$ and get
  \begin{equation}\label{eq:simplified:meanpart}
    \int_0^1 \phi^\varepsilon(x,y) \, dy = \frac{1}{\varepsilon}\int_0^1
    f^\varepsilon(x,y) \, dy  \,,
  \end{equation}
Taking the limit $\varepsilon \to 0$ in this equation and inserting (\ref{eq:phi=psi+zeta}), we get $\zeta(x) = \int_0^1 f^{(1)} (x,y) \, dy$, which determines $\phi^{(0)}$ uniquely. 

Now, if a standard numerical method is applied to (\ref{eq1}), (\ref{eqb}), the resulting matrix will be close, when $\varepsilon \ll 1$, to the singular matrix obtained from the discretization of (\ref{eq:simplified:degenerated}). Therefore, its condition number will blow up as $\varepsilon \to 0$, resulting in either low accuracy, or high computational cost. To overcome this problem, we decompose $\phi^{\varepsilon}$ according to
\begin{eqnarray}
& & \phi^{\varepsilon} = p^{\varepsilon} + q^{\varepsilon}, \qquad p^{\varepsilon} (x) = \int_{0}^{1} \phi^{\varepsilon}(x,y) \, dy , 
\label{eq:decom}
\end{eqnarray}
i.e. $p^{\varepsilon}$ is the average of $\phi^{\varepsilon}$ along straight lines parallel to $b$  and
 $q^{\varepsilon}$ is the fluctuation of the solution with respect to this average. $p^{\varepsilon}$ and $q^{\varepsilon}$ satisfy: 
\begin{eqnarray} 
& & \frac{\partial p^\varepsilon}{\partial y} (x,y) = 0, \quad  \forall (x,y) \in \Omega, \label{eq:peps_constant} \\
& & \int_{0}^{1} q^{\varepsilon}(x,y) \, dy =0, \quad \forall x \in [0,1]. 
\label{moy} 
\end{eqnarray}
They are orthogonal for the scalar product of $L^2$, i.e. $\int_{\Omega} p^\varepsilon q^\varepsilon \, dx \, dy = 0$.

Inserting this decomposition into \eqref{eq:simplified:meanpart}
yields 
\begin{eqnarray}
p^{\varepsilon}(x) = \frac{1}{\varepsilon} \int_{0}^1 \, f^\varepsilon (x,y)\, dy
 , \forall x \in [0,1] . 
 \label{eq:determin_peps}
\end{eqnarray}
Moreover, $p^{\varepsilon}$ satisfies  
$$\lim_{\varepsilon \to 0} p^\varepsilon(x)  =  \lim_{\varepsilon \to 0} \frac{1}{\varepsilon} \int_{0}^1 \, f^\varepsilon (x,y)\, dy = \int_{0}^1 \, f^{(1)} (x,y)\, dy = \zeta(x),$$ 
where $\zeta$ is defined by (\ref{eq:phi=psi+zeta}). Now, $q^{\varepsilon}$ is the solution of the following problem: 
\begin{eqnarray} 
\label{e1d}
- \frac{\partial^2}{\partial y^2 } q^{\varepsilon}(x,y) + \varepsilon
q^{\varepsilon}(x,y) &=&\xi^\varepsilon(x,y) , \hspace*{2 mm} \forall (x,y)
\in [0,1] \times ]0,1[ ,
\\ 
\label{e2d} 
\int_{0}^{1} q^{\varepsilon}(x,y) \, dy &=& 0,\hspace*{2 mm} \mbox{for} \hspace*{2 mm} x \in [0,1],
\\ 
\label{e3d} 
\frac{\partial}{\partial y} q^{\varepsilon}(x,y) &=& 0, \hspace*{2 mm} \mbox{for} \hspace*{2 mm} 
y=0 \; \mbox{or} \; y=1,
\end{eqnarray} 
where 
$$ \xi^\varepsilon = f^\varepsilon - \int_0^1 f^\varepsilon \,dy = f^\varepsilon - \varepsilon p^\varepsilon , $$ 
is the projection of $f^{\varepsilon}$ on the space of functions satisfying (\ref{moy}). Compared to (\ref{eq1}), (\ref{eqb}), system (\ref{e1d})-(\ref{e3d}) involves the additional condition (\ref{e2d}). This condition is important: it makes the system uniformly well-posed when $\varepsilon \to 0$. Additionally, the limit system is 
\begin{eqnarray} 
\label{e1d_e=0}
- \frac{\partial^2}{\partial y^2 } q^{(0)}(x,y) &=& f ^{(0)} , \hspace*{2 mm} \forall (x,y)
\in [0,1] \times ]0,1[ ,
\\ 
\label{e2d_e=0} 
\int_{0}^{1} q^{(0)}(x,y) \, dy &=& 0,\hspace*{2 mm} \mbox{for} \hspace*{2 mm} x \in [0,1],
\\ 
\label{e3d_e=0} 
\frac{\partial}{\partial y} q^{(0)}(x,y) &=& 0, \hspace*{2 mm} \mbox{for} \hspace*{2 mm} 
y=0 \; \mbox{or} \; y=1,
\end{eqnarray}  
and has a unique solution equal to $\tilde \psi$. Consequently, as $\varepsilon \to 0$
$$ \phi^\varepsilon = p^\varepsilon + q^\varepsilon \to \zeta + \tilde \psi = \phi^{(0)}. $$
Therefore, the proposed decomposition leads to two uniformly well-posed problems when $\varepsilon \to 0$, which allows to reconstruct the limit solution $\phi^{(0)}$ of the original problem.

The numerical approximations of conditions (\ref{eq:determin_peps}) or (\ref{e2d}) is delicate if the mesh is not aligned with the $y$ coordinate axis. In order to overcome this problem, a weak formulation is introduced. Define $V = H^1(0,1)$, $K = \{ v \in V \, | \, \partial_y v = 0 \}$. Then, $\phi^\varepsilon$ is the solution of the variational formulation
\begin{eqnarray}
&& \hspace{-1cm} \mbox{Find } \phi^\varepsilon \in V \mbox{ such that} \nonumber \\
&& \hspace{-1cm} \int_{\Omega} \frac{\partial \phi^\varepsilon}{\partial y}\, \frac{\partial \psi}{\partial y} \, dx \, dy + \varepsilon \int_{\Omega} \phi^\varepsilon \, \psi \, dx \, dy = \int_{\Omega} f^\varepsilon \, \psi \, dx \, dy, \quad \forall \psi \in V. \label{eq:var_form}
\end{eqnarray}
Let $K^\bot$ be the orthogonal space to $K$ in $L^2(0,1)$. Now, the decomposition (\ref{eq:decom}), corresponds to the decomposition of $\phi^\varepsilon$ on $K$ and $K^\bot$. Indeed, it is easily checked that  $p^\varepsilon \in K$ and $q^\varepsilon \in K^\bot$ and they are orthogonal, as already noticed. Now, inserting $\psi \in K$ in (\ref{eq:var_form}), we get that $p^\varepsilon$ is the solution of 
\begin{eqnarray}
&& \hspace{-2cm} \mbox{Find } p^\varepsilon \in K \mbox{ such that} \nonumber \\
&& \hspace{-2cm} \int_{\Omega} (p^\varepsilon - \frac{1}{\varepsilon} f^\varepsilon)  \, \psi \, dx \, dy =0, \, \forall \psi \in K, \label{eq:var_form_p}
\end{eqnarray}
which means that $p^\varepsilon$ is the orthogonal projection of $\varepsilon^{-1} f^\varepsilon$ onto $K$. Now, inserting $\psi \in K^\bot$ in (\ref{eq:var_form}) leads to 
\begin{eqnarray}
&& \hspace{-1.5cm} \mbox{Find } q^\varepsilon \in K^\bot \mbox{ such that} \nonumber \\
&& \hspace{-1.5cm} \int_{\Omega} \frac{\partial q^\varepsilon}{\partial y}\, \frac{\partial \psi}{\partial y} \, dx \, dy + \varepsilon \int_{\Omega} q^\varepsilon \, \psi \, dx \, dy = \int_{\Omega} (f^\varepsilon - \varepsilon p^\varepsilon) \, \psi \, dx \, dy, \, \forall \psi \in K^\bot, \label{eq:var_form_q}
\end{eqnarray}
which is the variational formulation of (\ref{e1d})-(\ref{e3d}). 

The use of these variational formulations allows for the discretization of (\ref{eq1}), (\ref{eqb}) on arbitrary meshes compared to the anisotropy direction. This is an important advantadge over the strong formulations (\ref{eq:determin_peps}) or (\ref{e1d})-(\ref{e3d}). These formulations are now generalized to arbitrary anisotropy fields $b$ in the next section.

\subsection{Presentation of the method for a general anisotropy field}
\label{sdec}

\subsubsection{Preliminaries}
\label{subsubsec_present_prelim}

This subsection is devoted to the resolution of degenerate elliptic problems (\ref{P1}), (\ref{bc1}) for general anisotropy fields $b$. we first introduce the space 
\begin{eqnarray*}
\mathcal{V} &=& \{ \phi \in L^2 (\Omega ) \, / \nabla \cdot (b \phi ) \in L^2 (\Omega ) \}, \\
 K &=& \{ \phi \in \mathcal{V} \, / \, \nabla \cdot (b \phi ) = 0 \; \mbox{on} \; \Omega \},
\\ 
\mathcal{W} &=& \{ h \in L^2(\Omega) \, / (b \cdot \nabla )h \in   L^2(\Omega) \},
\\ 
\mathcal{W}_0 &=& \{ h \in  \mathcal{W} \, / \, (b \cdot \nu ) h = 0 \; \mbox{on} \; \partial{\Omega}\}.
\end{eqnarray*}
The projection of a function on $K$ is the generalization of the average operation (\ref{eq:determin_peps}), while the projection on $K^\bot$ corresponds to computing its fluctuation. The space $\mathcal{W}_0$ is used to characterize $K^\bot$. The projections on $K$ and $K^\bot$ are well-defined thanks to the:

\begin{theorem} \label{fondamental}
We have the following properties
\begin{enumerate}
\item ${K}$ is closed in $L^{2}(\Omega )$.
\item $\mathcal{W}_0$ equiped with the norm $\| \, h \, \|_{\mathcal{W}_0} = \| \, (b \cdot \nabla ) h \, \|_{L^{2}(\Omega )}$ is a Hilbert space and $(b \cdot \nabla ) \mathcal{W}_0$ 
is a closed space of $L^{2}(\Omega )$.
\item $ K^{\bot}= (b \cdot \nabla ) \mathcal{W}_0$.  
 \end{enumerate}
\end{theorem}

\noindent
\begin{proof} $1)$ Let $\phi_n \in {\mathcal V}$ such that  $\phi_n \to \phi$ in $L^2(\Omega)$. Then, $\phi_n \to \phi$ in the distributional sense and the operation $\phi \to \nabla \cdot (b \phi)$ is continuous for the topology of distributions. Therfore, $\nabla \cdot (b \phi) = 0$, which shows that $\phi \in {\mathcal V}$. 

\medskip
\noindent
$2)$ $\mathcal{W}_0$ is a Hilbert space for the norm $\| \, h \, \| = 
\| \, h \, \|_{L^{2}(\Omega )} + \| \, (b\cdot \nabla ) h \, \|_{L^{2}(\Omega )} $. According to 
the Poincar\'e inequality, the norms $\| \hspace*{2mm} \| $ and $\| \hspace*{2mm} \|_{\mathcal{W}_0}$ are 
equivalent. The closedness of $\mathcal{W}_0$ for the $L^2$ topology follows from $3)$. 

\medskip
\noindent 
$3)$ The inclusion $   (b \cdot \nabla ) \mathcal{W}_0   \subseteq {K}^{\bot}$ is obvious. 
We sketch the proof of the converse inclusion and leave the details to the reader. We make the hypothesis that all $b$-field lines are either tangent to a non-zero measure set of $\partial \Omega$ or intersect $\partial \Omega$ at two points $x_-$ and $x_+$ such that $\pm (b \cdot \nu)(x_\pm) >0$. The points $x_-$ and $x_+$ are called the conjugate points of the $b$-field line and are respectively the incoming and outgoing points of this field line to the domain. These assumptions can certainly be weekened at the expense of technical difficulties which are outside the scope of this paper. Let $\psi \in K^\bot$. By taking the primitive of $\psi$ along the $b$-field lines, there exists $\phi \in {\mathcal W}$ such that $\psi = (b \cdot \nabla) \phi$. We can additionally impose that $\phi = 0$ on $\partial \Omega_-$ where  $\partial \Omega_\pm= \{ x \in \partial \Omega \, | \,  \pm (b \cdot \nu)(x) >0 \}$. Let $\theta \in K$. We have 
\begin{equation}  0 = \int_{\Omega} \psi \, \theta \, dx = \int_{\Omega} (b \cdot \nabla) \phi \, \theta \, dx = \int_{\partial \Omega} (b \cdot \nu) \phi \, \theta \, dS(x) , 
\label{eq:conjugate}
\end{equation}
where $dS(x)$ is the superficial measure on $\partial \Omega$. Since, $\theta \in K$ its values at conjugate points are related by a linear relation. In particular, they can be taken simultaneously non-zero. Then, since the values of $\phi$ on $\partial \Omega_-$ vanish, (\ref{eq:conjugate}) implies that the values of $\phi$ on $\partial \Omega_+$ vanish as well. Consequently, $(b \cdot \nu) \phi = 0$ on $\partial \Omega$, which shows that $\phi \in {\mathcal W}_0$. This proves the result. 
\end{proof}

\medskip
\noindent 
Therefore, we can decompose $\phi^{\varepsilon}$ uniquely as
\begin{eqnarray} \label{dec}
 \phi^{\varepsilon} = p^{\varepsilon} + q^{\varepsilon}, \quad p^{\varepsilon} \in K , \quad q^{\varepsilon} \in K^\bot,
\end{eqnarray}
and state problem (\ref{P1}), (\ref{bc1}) as
\begin{eqnarray}
\label{P2} &  & -\left(  b \cdot \nabla \right) \left( \nabla \cdot (b \, q^{\varepsilon} ) \right) + \varepsilon 
(p^{\varepsilon} + q^{\varepsilon} ) = f ^{\varepsilon} , \hspace*{2
  mm} \mbox{in} \hspace*{2 mm} \Omega ,
\\ \label{bc2} & & \left( b \cdot \nu \right) \nabla \cdot \left( b \,
  q^{\varepsilon } \right) = 0, \hspace*{2
  mm} \mbox{in} \hspace*{2 mm} \partial \Omega ,
\\  &  &  p^{\varepsilon} \in K \; \text{and } \;  q^{\varepsilon} \in K^{\bot}.
\end{eqnarray}
Next, we introduce the variational approach. We multiply (\ref{P2}) by a test function 
$\psi \in \mathcal{V}$, and integrate it on $\Omega$. Using a Green formula together with the 
boundary condition (\ref{bc2}), we find that  
\begin{eqnarray} \label{var1}
\int_{\Omega}  \nabla \cdot \left( b \, q^{\varepsilon} \right) \, \nabla \cdot \left( b \, \psi \right) dx
+ \varepsilon \int_{\Omega}
(p^{\varepsilon} + q^{\varepsilon} ) \psi dx = \int_{\Omega }  f^{\varepsilon} \psi  dx.
\end{eqnarray}
The aim now is to decompose problem (\ref{var1}) into a problem for $p^{\varepsilon}$ 
and a problem for $q^{\varepsilon}$. Hence in the following two subsections the test
function $\psi$ is chosen successively in $K$ and in $K^\bot.$

\subsubsection{Equation for $p^{\varepsilon} \in K  $} 
\label{pe}

Chosing $\psi = r \in K$ in (\ref{var1}), we obtain the problem
\begin{eqnarray} 
& & \hspace{-1cm} \mbox{Find } p^\varepsilon \in K \mbox{ such that } \, \, 
\int_{\Omega} \left( \varepsilon p ^{\varepsilon} - f^{\varepsilon} \right) r \, dx =0, \,  \forall r \in K . \label{p}
\end{eqnarray} 
This problem admits a solution in $K$ which is uniformly bounded in $L^2(\Omega)$ as $\varepsilon \to 0$ under the compatibility condition
\begin{equation} \label{comp}
\lim_{\varepsilon \to 0} \left( \frac{1}{\varepsilon} \int_{\Omega} f^\varepsilon \, r \, dx \right) \mbox{ exists and is finite}, \quad \forall r \in K. 
 \end{equation}
Assuming  that $f^\varepsilon$ has the following decomposition
 \begin{eqnarray*}
 f^{\varepsilon} = f^{(0)} + \varepsilon f^{(1)} + o(\varepsilon) \,.
\end{eqnarray*} 
in $L^2(\Omega)$, this condition implies that $f^{(0)} \in K^{\bot}$. 
Next, since $\varepsilon p^{\varepsilon} - f^{\varepsilon} \in K^\bot$, according to 
Theorem \ref{fondamental} there exists $g^{\varepsilon} \in \mathcal{W}_0$ 
 such that 
\begin{eqnarray} \label{rel}
  \varepsilon p^{\varepsilon} - f^{\varepsilon} = (b \cdot \nabla ) g^{\varepsilon }.   
  \end{eqnarray}
Taking the product with $b$ and the divergence of the result, we
obtain the following 

\begin{proposition}
$p^\varepsilon$ is given by 
\begin{eqnarray}
& & \hspace{-1cm} p^{\varepsilon} = \frac{1}{\varepsilon} \left( f^{\varepsilon} + b
  \cdot \nabla g^{\varepsilon}  \right) \hspace*{3 mm} \mbox{in} \hspace*{3 mm} \Omega. \label{bfp}
\end{eqnarray}
where $g^{\varepsilon}$ satisfies the problem: 
\begin{eqnarray}
& & \hspace{-1cm} - \nabla \cdot \left( (b \otimes b) \nabla g^{\varepsilon} \right) = 
\nabla \cdot \left( f^{\varepsilon} b    \right)  \quad \mbox{in } \Omega,  \label{fp} \\
& & \hspace{-1cm} (b \cdot \nu) g^{\varepsilon} = 0 \quad \mbox{on } \partial \Omega,   \label{fp_bc}
\end{eqnarray}  
or, in variational form 
\begin{eqnarray} 
& & \hspace{-1cm} \mbox{Find } g^\varepsilon \in \mathcal{W}_0 \mbox{ such that } \nonumber \\
& & \hspace{-1cm} \int_{\Omega} (b \cdot \nabla g^\varepsilon) (b \cdot \nabla \theta) = \int_{\Omega} f^{\varepsilon} b \cdot \nabla \theta  \, dx , \,  \forall \theta \in \mathcal{W}_0 . \label{p_variational}
\end{eqnarray} 
\label{prop_peps}
\end{proposition}

\subsubsection{Equation for $q^{\varepsilon}\in K ^{\bot}$}
\label{subsubsec_qe}

Taking $\psi = s \in K^{\bot}$ in (\ref{var1}) gives:
\begin{eqnarray} \label{varq}
\int_{\Omega} {\nabla } \cdot ( b q^{\varepsilon} ) \nabla  \cdot (bs ) \,
dx
+ \varepsilon \int_{\Omega}q^{\varepsilon}s \, dx 
= \int_{\Omega} f^{\varepsilon} s \, dx. 
\end{eqnarray}
But since $q^{\varepsilon}$ and $s \in K^{\bot}$, theorem 
\ref{fondamental} implies that there exists $h^{\varepsilon}$ and $\theta \in \mathcal{W}_0$ 
such that $q^{\varepsilon} = b \cdot \nabla h^{\varepsilon}$ and $s = b \cdot \nabla \theta$. Therefore, we get the following

\begin{proposition}
$q^{\varepsilon}$ is given by: 
\begin{eqnarray} 
& & \hspace{-1cm} q^{\varepsilon} = b \cdot { \nabla } h^{\varepsilon}, 
\label{hfi}
\end{eqnarray}
where $h^{\varepsilon}$ satisfies the following fourth-order problem: 
\begin{eqnarray} \label{bi1}
  -  \nabla \cdot \left[ ( b \otimes b ) \nabla ( \nabla \cdot (b \otimes b) \nabla h^\varepsilon ) \right] 
+ \varepsilon  \nabla \cdot \left( (b \otimes b  ) { \nabla } h^{\varepsilon} \right) =
{  \nabla } \cdot (b f^{\varepsilon} ) , 
\hspace*{1 mm} \mbox{in} \hspace*{1 mm} \Omega , \hspace*{5 mm}
\\  \label{bi2}  (b \cdot \nu )  { \nabla } \cdot \left( (b \otimes b )
  \nabla h^{\varepsilon} \right)  = 0 , \; 
\mbox{on} \;   \partial \Omega , \;\hspace*{56 mm}
\\   \label{bi3} (b \cdot \nu ) h^{\varepsilon} = 0 , \; \mbox{on} \;   \partial
\Omega ,  \;\hspace*{79 mm}
\end{eqnarray} 
or, in variational form 
\begin{eqnarray} 
& & \hspace{-1cm} \mbox{Find } h^\varepsilon \in \mathcal{W}_0 \mbox{ such that } \nonumber \\
& & \hspace{-1cm} 
\int_{\Omega} \nabla \cdot \left( (b \otimes b)  \nabla  h^{\varepsilon}
\right) \,  
{  \nabla } \cdot \left( (b \otimes b) { \nabla } \theta  \right) dx +
\varepsilon \int_{\Omega} (b \cdot {  \nabla }  h^{\varepsilon}) 
\, (b\cdot { \nabla } \theta ) \, dx = \nonumber \\
& & \hspace{7cm} 
= \int_{\Omega} f^{\varepsilon} \, 
(b \cdot {  \nabla }  \theta )\, dx , 
\label{eq:qeps_variational}
\end{eqnarray} 
\label{prop_qeps}
\end{proposition}

\medskip
\noindent
The resolution of problem (\ref{P1}), (\ref{bc1}) can be summarized in the following proposition. 

\begin{proposition} 
\label{sumhom}
If $f^{\varepsilon}$ satisfies
(\ref{comp}), problem (\ref{P1}), (\ref{bc1}) is formally equivalent to the two
problems (\ref{bfp}), (\ref{fp}), (\ref{fp_bc}) on the one hand and (\ref{hfi}), (\ref{bi1}), (\ref{bi2}), (\ref{bi3}).
\end{proposition}

\begin{remark}
In \cite{DDLNN_sub}, the characterization of $K^\bot$ as $(b \cdot \nabla ) \mathcal{W}_0$ is not used. Instead, the constraint that $q \in K^\bot$ is taken into account through a mixed formulation. The number of unknowns and the size of the problem are therefore larger in \cite{DDLNN_sub} than in the present work. In practice, the resolution of the fourth order problem (\ref{bi1}), (\ref{bi2}), (\ref{bi3}) can be reduced by solving two second-order problem, as shown below. Therefore, the introduction of a fourth order problem does not bring specific difficulties.  
\label{rem:compar_liter}
\end{remark}

\subsubsection{Extension to non-homogeneous Neumann boundary conditions}
\label{bnh}

The application targeted in this paper, and detailed in
section~\ref{DA2},  requires the handling of non-homogeneous Neumann boundary conditions. 
In this subsection $\phi^{\varepsilon}$ is solution to the following 
inhomogeneous Neuman problem: 
\begin{eqnarray} \label{sysr}
\varepsilon \phi^{\varepsilon} - 
 (b \cdot \nabla ) (\nabla \cdot (b \, \phi^{\varepsilon})) 
&=& b \cdot \nabla \kappa + f^\varepsilon_2, \hspace*{3 mm} \mbox{on} \hspace*{3 mm} \Omega,
\\ \label{bc} (b \cdot \nu ) \nabla \cdot (b \, \phi^{\varepsilon}) &=& -(b \cdot
\nu ) \kappa, \hspace*{3 mm} \mbox{on} \hspace*{3 mm} \partial \Omega.
\end{eqnarray}
where $\kappa$ is a given function in ${\mathcal W}$. We denote by $f_1 = b \cdot \nabla \kappa$ and by $f^\varepsilon = f_1 + f^\varepsilon_2$. 

Using the same decomposition (\ref{dec}) as before, we find that $p^\varepsilon$ satisfies (\ref{bfp}) and $g^\varepsilon$ is the solution of (\ref{fp}), (\ref{fp_bc}) or (\ref{p_variational})
with $f^\varepsilon$ replaced by $f^\varepsilon_2$ (and satisfying (\ref{comp})). Similarly, $q^\varepsilon$ satisfies (\ref{hfi}) and $h^\varepsilon$ is the solution of (\ref{bi1}), (\ref{bi2}), (\ref{bi3}), or of (\ref{eq:qeps_variational}) with '0' at the right-hand side of (\ref{bi2}) replaced by $(b \cdot \nu) \kappa$, the other terms being unchanged. The details are left to the reader.

\setcounter{equation}{0}
\section{Space discretization}
\label{num}

The problem is discretized using a finite volume method.
The domain is decomposed into a familly $\mathcal{R}$ of rectangles $M_{i-{1}/{2}, j-{1}/{2}} = ]x_{i-1},x_i[  \times ]y_{j-1},y_j[$ with $x_i = i \Delta x$  and $y_j =j \Delta
y$. We look for a piecewise constant approximation $p_{\mathcal{R}}^{\varepsilon}$ of $p^{\varepsilon}$ on each
$M_{i-{1}/{2}, j-{1}/{2}}$ and denote by $p_{i-{1}/{2},j-{1}/{2}}$ its constant value on this rectangle. The function $g^{\varepsilon}$ is approximated by a constant function on a dual mesh $\mathcal{D}$, 
consisting of rectangles $ \mathcal{D}_{i,j}= ]  x_{i-{1}/{2}} , x_{i+{1}/{2}} [  \times ]  y_{j- {1}/{2}} , y_{j+{1}/{2}} [ $
where $x_{i-{1}/{2 }} = (i - {1}/{2}) \Delta x$, $y_{i-{1}/{2 }} = (i - {1}/{2}) \Delta y$. Then $g^{\varepsilon}$ is approximated by a piecewise constant function $g_{\mathcal{D}}^{\varepsilon}$ with its constant values denoted by  $g_{i,j}^{\varepsilon}$. We approximate (\ref{bfp}) by
\begin{eqnarray*}
p_{i-\frac{1}{2}, j-\frac{1}{2}} =  \frac{1}{\varepsilon} \left(
  f^{\varepsilon}(x_{i-\frac{1}{2}} , y_{j -\frac{1}{2}}) +
  b(x_{i-\frac{1}{2}} , y_{j -\frac{1}{2}}) \cdot 
( \nabla g^{\varepsilon})_{i - \frac{1}{2},j-\frac{1}{2}} 
 \right).
\end{eqnarray*}

We now define approximations $\left(  b \cdot \nabla  \right)_{\mbox{\scriptsize{app}}}$ and 
$\nabla \cdot ( \hspace*{2 mm} \cdot \hspace*{2 mm} b )_{\mbox{\scriptsize{app}}}$ of operators 
$ \Psi \mapsto ( b \cdot \nabla \Psi  ) $ and $\Phi \mapsto 
\nabla \cdot ( b \, \Phi    ) $ such that they are 
discrete dual operators to each other. For this purpose, we define $L_{\mathcal R}$ and $L_{\mathcal D}$ the space of piecewise constant functions on meshes of types ${\mathcal R}$ and ${\mathcal D}$ respectively.

\begin{definition} 
The operator $\left(  b \cdot \nabla  \right)_{\mbox{\scriptsize{app}}}$: $L_{\mathcal{D}} \rightarrow
L_{\mathcal{R}}$  is defined by
\begin{equation}
  \begin{split}
    (\left(   b \cdot \nabla \Psi
    \right)_{\mbox{\scriptsize{app}}})_{i-\frac{1}{2}, j-\frac{1}{2}} = \hspace*{1cm} &\\
    = b(x_{i- \frac{1}{2}} , y_{j-\frac{1}{2}}) \cdot
    \bigg(  & \Big( \frac{\Psi_{i,j}
      - \Psi_{i-1,j}}{2 \Delta x} +   \frac{\Psi_{i,j-1} 
      - \Psi_{i-1,j-1}}{2 \Delta x}  \Big) ,    
    \\ 
    &\Big( \frac{\Psi_{i,j} - \Psi_{i,j-1}}{2 \Delta y} 
    +   \frac{\Psi_{i-1,j} - \Psi_{i-1,j-1}}{2\Delta
      y}        \Big)
    \bigg).
\end{split} \label{grad}
\end{equation}

\noindent
The operator $\nabla \cdot ( \hspace*{2 mm} \cdot \hspace*{2 mm} b )_{\mbox{\scriptsize{app}}}$:  $L_{\mathcal{R}} \rightarrow
L_{\mathcal{D}}$ is defined by
\begin{equation}\label{div} 
\begin{split}
(\nabla \cdot (b \Phi)_{\mbox{\scriptsize{app}}})_{i,j} =\\
=&  \left( \frac{1}{2 \, \Delta x} b_x ( x_{i + \frac{1}{2}} , y_{j-
    \frac{1}{2}}   ) -    \frac{1}{2 \, \Delta y} b_y ( x_{i + \frac{1}{2}} , y_{j-
    \frac{1}{2}}   )  \right) \Phi_{i + \frac{1}{2} ,j-
    \frac{1}{2}}   \\
+& \left( \frac{1}{2 \, \Delta x} b_x ( x_{i + \frac{1}{2}} , y_{j+
    \frac{1}{2}}   ) +    \frac{1}{2 \, \Delta y} b_y ( x_{i + \frac{1}{2}} , y_{j+
    \frac{1}{2}}   )  \right) \Phi_{i + \frac{1}{2} ,j+
    \frac{1}{2}} \\  
- & \left( \frac{1}{2 \, \Delta x} b_x ( x_{i - \frac{1}{2}} , y_{j-
    \frac{1}{2}}   ) +    \frac{1}{2 \, \Delta y} b_y ( x_{i - \frac{1}{2}} , y_{j-
    \frac{1}{2}}   )  \right) \Phi_{i - \frac{1}{2} , j-
    \frac{1}{2}}  \\ 
 -& \left( \frac{1}{2 \, \Delta x} b_x ( x_{i - \frac{1}{2}} , y_{j+
    \frac{1}{2}}   ) -    \frac{1}{2 \, \Delta y} b_y ( x_{i - \frac{1}{2}} , y_{j+
    \frac{1}{2}}   )  \right) \Phi_{i - \frac{1}{2} ,j+
    \frac{1}{2}}   .
\end{split} 
\end{equation}
\end{definition}

\begin{proposition} \label{dual}
$ \big(b \cdot \nabla \big)_{\mbox{\scriptsize{app}}} $ and $\nabla \cdot \big( b \, \cdot
  \,\big)_{\mbox{\scriptsize{app}}} $ are adjoint operators to each other .
\end{proposition}

\begin{proof}
Easy and left to the reader, thanks to a discrete Green formula. 
\end{proof}

\medskip
\noindent
Next, we define $\left(\nabla \cdot ( (b \otimes b) \cdot
\nabla)
\right)_{\mbox{\scriptsize{app}}}$ by the composition  of the two operators
$ \big(b \cdot \nabla \big)_{\mbox{\scriptsize{app}}} $ and $\nabla \cdot \big( b \, \cdot
 \,\big)_{\mbox{\scriptsize{app}}} $: 
 
\begin{definition} We define: 
\begin{equation}\label{eq:ellipticop:consitent}
\left(\nabla \cdot ( b \otimes b \cdot
\nabla)
\right)_{\mbox{\scriptsize{app}}}
=\left(\nabla \cdot ( \cdot \hspace*{3 mm} b 
)\right)_{\mbox{\scriptsize{app}}} \circ \left(b \cdot \nabla \right)_{\mbox{\scriptsize{app}}} ,
\end{equation}
where $\circ$ is the composition operation. 
\label{def_ope_sec_order}
\end{definition}

\medskip
\noindent
Finally, the approximation of problem (\ref{fp}), (\ref{fp_bc}) is by solving the discrete problem for the piecewise constant function $g$ on ${\mathcal D}$: 
\begin{eqnarray} \label{discret}
 \left(\nabla \cdot ( b (\otimes) b \cdot
\nabla)
\right)_{\mbox{\scriptsize{app}}}g 
= \left( \nabla \cdot ( b  f )\right)_{\mbox{\scriptsize{app}}}
,
\end{eqnarray}
together with Dirichlet boundary conditions on $g$, where $f$ is a piecewise constant function on ${\mathcal R}$. 

Now, problem (\ref{bi1}), (\ref{bi2}), (\ref{bi3}) for $q^\varepsilon$  can be decomposed in two decoupled second-order elliptic problems of the type (\ref{fp}), (\ref{fp_bc}) and can be solved by a similar method. Indeed by setting $   u = - \nabla \cdot  \left(  ( b \otimes b) \nabla h \right)  $, we get that (\ref{bi1}), (\ref{bi2}), (\ref{bi3}) is equivalent to the following two elliptic problems:
\begin{eqnarray} \label{el1}
 \nabla \cdot \left( (  b \otimes  b)  \nabla  u \right) - \varepsilon u &=& \nabla \cdot ( bf   ) \hspace*{5 mm} \mbox{in} \hspace*{5mm} \Omega
\\ \label{b1} ( b \cdot  \nu ) u &=& 0 \hspace*{5 mm} \mbox{on} \hspace*{5mm} \partial \Omega
\end{eqnarray}
and 
\begin{eqnarray}
\label{el2} -  \nabla \cdot \left( (  b \otimes  b )  \nabla h \right) &=& u  \hspace*{5 mm} \mbox{in} \hspace*{5mm} \Omega
\\ \label{b2} ( b \cdot  \nu ) h &=& 0 \hspace*{5 mm} \mbox{on}
\hspace*{5mm} \partial \Omega .
\end{eqnarray}

To summarize, the resolution of problem (\ref{P1}), (\ref{bc1}) reduces to three independent resolutions of problems similar to (\ref{discret}).

\setcounter{equation}{0}
\section{Application to the Euler-Lorentz system in the drift
  limit}
\label{DA2}

\subsection{Introduction}
\label{DA2_intro}

In this section the drift-fluid limit of the isothermal
  Euler-Lorentz is investigated. This regime is representative of
  strongly magnetized plasma, for which the pressure term equilibrates
  the Lorentz force. It is obtained by letting a
  dimensionless parameter $\varepsilon$, representing the non-dimensional
  gyro-period as well as the square Mach number, go to zero. This limit
  is singular because the
  momentum equation in the direction of the magnetic field degenerates. Since the field may not be uniformly large, we wish to derive an Asymptotic- Preserving (AP)
  scheme which guarantees accurate discretizations of both the limit
  regime for strongly magnetized plasma  ($\varepsilon \ll 1$) and 
  the standard Euler-Lorentz system when the
  field strength is mild ($\varepsilon \sim 1$). 
  With this aim, the Euler-Lorentz system
  is discretized in time by a semi-implicit scheme. 
  
  This scheme has already been studied in \cite{DDSV_JCP09} for a uniform and constant magnetic field aligned with one coordinate
and for physically less meaningful Dirichlet boundary  conditions. 
The present methodology allows us to investigate the case of non-uniform magnetic fields and Neumann boundary conditions. Indeed, the 
  anisotropic elliptic equation (\ref{P1}), (\ref{bc1}) appears as the central building block of the scheme, which allows for the
  computation of the field-aligned momentum component. In this presentation, we will mainly focus on this aspect, the other ones being unchanged compared to \cite{DDSV_JCP09}.

\subsection{The Euler-Lorentz model and the AP scheme}
\label{ss_DA2}

The scaled isothermal Euler-Lorentz model takes
the form: 
\begin{eqnarray}
  &&\partial_t{n_{\varepsilon}} +  \nabla \cdot
  \left(n_{\varepsilon}  u_{\varepsilon}\right) = 0\,,\label{ec}\\
  &&\varepsilon \Big[\partial_t{\left(n_{\varepsilon}
      u_{\varepsilon}\right)} +  \nabla \cdot\left(n_{\varepsilon} 
     u_{\varepsilon} \otimes  u_{\varepsilon}\right)\Big] +
  T\, \nabla n_{\varepsilon} = n_{\varepsilon} \left( E + 
  u_{\varepsilon} \times  B \right)\,,\label{em}
\end{eqnarray} 
where $n_{\varepsilon}$, $ u_{\varepsilon}$ and $T$ are the density, 
the velocity and the temperature of the  ions, respectively. Here, the electric field $ E$ and the magnetic field
$ B$ are assumed to be given functions. The parameter $\varepsilon$ is related to the gyro-period of the particles about the magnetic field lines, and simultaneously to the squared Mach number. We refer to \cite{DDSV_JCP09} for more details on the model, the scaling and the drift-fluid limit $\varepsilon \to 0$. 

Now we introduce the time dicretization of the model. 
 Let $ B^m$ be the magnetic field
at time~$t^m$, $|B|^m$ its magnitude and ${ b^m}= { B^m}/|B|^m$ its
direction. For a given vector field $v$, denote by
$(v)_\parallel^m$ and $( v)_\bot^m$ its parallel and perpendicular 
components with respect to $ b^m$ ie   
\begin{eqnarray*}
v = (v)_\parallel^m \,b^m + ( v)_\bot^m, \hspace*{3 mm}
(v)_\parallel^m = v \cdot b^m, \hspace*{3 mm} ( v)_\bot^m = b^m \times (v \times b^m).
\end{eqnarray*}
 Similarly, we denote by
$\nabla_\parallel^m$ and $\nabla_\parallel^m \cdot$\, the parallel
gradient and divergence operators respective to this field. The time semi-discrete scheme proposed in \cite{DDSV_JCP09} is as follows: 

\begin{definition}\label{def:schema:AP}
The AP scheme is the scheme defined by:
\begin{eqnarray}
  &&\frac{n^{m+1} - n^m}{\Delta t} +   \nabla \cdot \left(n
  u\right)^{m+1} = 0\,,\label{DA5e1}\\
  &&\varepsilon \Big[\frac{\left(n u \right)^{m+1}-\left(n u\right)^m}{\Delta t} +  \nabla \cdot \left(n  u \otimes
     u \right)^m\Big] + T\left(  \nabla n^{\#}\right)^{m+1}
  \notag\\
 && \hspace{2.0cm}= n^m  E^{m+1} + \left(n u \right)^{m+1}
  \times  B^{m+1}\,,\label{DA5e2}
\end{eqnarray} 
where $\left( \nabla n^{\#}\right)^{m+1}$ is given by,
 \begin{equation}\label{DA5e3}
   \left( \nabla  n^{\#}\right)^{m+1} = \left( { \nabla }
   n^m\right)_\perp^{m+1} + \left(  \nabla
   n^{m+1}\right)_{\parallel}^{m+1} \,{ b}^{m+1}\,.
 \end{equation}
\end{definition}

\medskip
\noindent
By considering the scalar product of \eqref{DA5e2} with ${
b}^{m+1}$, we get
\begin{eqnarray*}
  \varepsilon (  \frac{ \left(n  u\right)^{m+1}-\left(n {
          u}\right)^{m}  }{\Delta t} 
+  \nabla \cdot (n_{\varepsilon}  u_{\varepsilon} \otimes 
u_{\varepsilon} )^{m} ) \cdot b^{m+1}            
\\ = -T  \nabla ^{m+1} n^m \cdot b^{m+1} + n^m  E^{m+1}  \cdot
    b^{m+1}
 \end{eqnarray*}
 and after easy computations \cite{DDSV_JCP09}, we find that $\left(n  u\right)^{m+1}_{\parallel} $ satisfies the following  anisotropic elliptic problem: 
\begin{equation}\label{DA5e14}
  \begin{aligned}
    &\frac{\varepsilon}{\Delta t}\left(n  u\right)^{m+1}_{\parallel} 
    - T\,\Delta t \,\nabla_{\parallel}^{m+1}
    \Big(\nabla^{m+1}_{\parallel} \cdot \left( (n
     u)^{m+1} \right)^{m+1}_{\parallel} \Big) \\   
    &\hspace{1.0cm}=T\,\Delta t \, \nabla_{\parallel}^{m+1}\Big( \nabla \cdot
    \left( (n  u)^{m+1} \right)^{m+1}_{\perp} \Big)
    -T\,\nabla_{\parallel}^{m+1} n^m\\  
    &\hspace{1.1cm}+\Big[\frac{\varepsilon}{\Delta
    t}\left(n  u\right)^m - \varepsilon \Big( \nabla \cdot
    \left(n  u \otimes  u\right)^m \Big) + n^m  E^{m+1}
    \Big]_\parallel^{m+1} \,.  
\end{aligned}
\end{equation}
By setting $(n  u)_{\parallel}^{m+1} =
\phi^{\varepsilon}$ and by taking $f = f_1 +f_2$ with
\begin{eqnarray}
f_1 &=& \frac{1}{\Delta t}  b \cdot \nabla (\nabla \cdot (n 
  u_{\perp}^{m+1})) , \label{eq:EL_f1}
\\ f_2 &=& - \Big[ \frac{\varepsilon}{T \, (\Delta t)^2} (n \,    u)^{m}
- \frac{\varepsilon }{T \, \Delta t} \,  \nabla \cdot 
(n \, { u} \otimes { u} )^{m} + n^m E^{m+1}
\Big ]_{\parallel}^{m+1}
\\ &-& \frac{1}{\Delta \, t} ( b \cdot \nabla n^m ). \label{eq:EL_f2}
\end{eqnarray}
this problem can be put in the framework of (\ref{P1}). In \cite{DDSV_JCP09}, because $b$ was chosen parallel to one of the coordinate axes, a direct discretization of (\ref{DA5e14}) using finite differences could be performed. Here, for an arbitrary anisotropy direction $b$, we use the method developed in the previous sections. We do not detail the description of the discretization of the other equations, since it follows \cite{DDSV_JCP09}. 

The right-hand side (\ref{eq:EL_f2}) can be decomposed 
as $f_2^\varepsilon =
  f_2^{(0)} + \varepsilon  f_2^{(1)} $
with $f_2^{(0)}$ corresponding to the first two terms and $f_2^{(1)}$, to the last two one. 
Moreover if we suppose that 
\begin{eqnarray}
\Big[ n^m E^{m+1}\Big ]_{\parallel}^{m+1}
- \frac{1}{\Delta \, t} ( b \cdot \nabla n^m ) \, \in \, K^{\perp} ,
\label{eq:EL_comp}
\end{eqnarray}
the compatibility condition (\ref{comp}) is satisfied. This property amounts to saying that the integrated force along a magnetic field line is zero. If the property is not satisfied, parallel velocities of order $O(1/\varepsilon)$ are generated, which is physically unrealistic (because collisions will ultimately slow down the plasma ions). Therefore, assuming (\ref{eq:EL_comp}) is physically justified.

As in \cite{DDSV_JCP09}, we will compare the AP scheme with 
the classical semi-discrete scheme for the Euler-Lorentz model, given by:

\begin{definition}\label{def:schema:class}
The 'classical' semi-discrete scheme is defined by: 
\begin{eqnarray}
&&\frac{n^{m+1} - n^m}{\Delta t} +   \nabla \cdot \left(n{
  u}\right)^{m} = 0\,,\label{class1}\\
&&\varepsilon \Big[\frac{\left(n  u  \right)^{m+1}-\left(nu\right)^m}{\Delta t} +  \nabla \cdot \left(n { u} \otimes
     u \right)^m\Big] + T\left(  \nabla n \right)^{m}
  \notag\\
  &&\hspace{2.0cm}= n^m E^{m+1} + \left(n u \right)^{m+1}
  \times  B^{m+1}\,.\label{class2}
\end{eqnarray}
\end{definition}

\medskip
\noindent
In \cite{DDSV_JCP09}, it is shown that this scheme is not uniformly stable
with respect to $\varepsilon$ and so that it cannot be AP. 

Except from the parallel momentum equation, which has just been discussed, the other equations of the model are discretized following \cite{DDSV_JCP09}. For the sake of brevity, we will not reproduce their presentation here.

\subsection{Boundary conditions}
\label{ss_DA6}

The following boundary conditions are set up for test purposes only. 
We impose Dirichlet boundary conditions on the density $n^{m+1}
=n_B$ with $n_B$ independent of time. For the perpendicular momentum, we impose 
the relation obtained after taking the limit when $\varepsilon \to 0$ in (\ref{em}),
\begin{eqnarray*}
n  u_{\perp}^{m+1} = - \frac{1}{|B|^{m+1}}  b \times \left( 
T \,  \nabla n^m + n^m  E ^{m+1} \right).
\end{eqnarray*}
By considering the mass conservation equation at the domain boundary, we
have 
\begin{eqnarray*}
\frac{n^{m+1}-n^m}{\Delta t} + \nabla \cdot (b \, n^{m+1}  u_{\parallel}^{m+1} ) 
+ \nabla \cdot (n u)_{\perp}^{m+1}  = 0, \hspace*{3 mm}
\mbox{on} \hspace*{3 mm} \partial \Omega .
\end{eqnarray*}
Therefore, as the density satisfies Dirichlet boundary conditions with time-independent Dirichlet values, 
we get 
\begin{eqnarray*}
( b \cdot {\nu})  \nabla \cdot (b \, n u_{\parallel}^{m+1}  ) 
= - (b \cdot \nu) \nabla \cdot ( n \, u_{\perp}^{m+1} ),\hspace*{3 mm}
\mbox{on} \hspace*{3 mm} \partial \Omega .
\end{eqnarray*}
Therefore, $n u_{\parallel}^{m+1} $ is a solution to the anisotropioc elliptic
problem with inhomogeneous Neumann boundary 
conditions (\ref{sysr}), (\ref{bc}), with $\kappa = - (b \cdot \nu)
\nabla \cdot ( n \, u_{\perp}^{m+1} ).$
Then, we can apply the framework of section~\ref{bnh}.
When $n u_{\parallel}$ has been calculated, an approximation is employed in
order to provide values of $n u_{\parallel}$ in a layer of fictious cells surrounding the boundary, by using homogeneous Neumann boundary conditions. The values in the fictitious cells are then useful to compute gradient terms which occur in the other equations of the Euler-Lorentz model.

\setcounter{equation}{0}
\section{Numerical results for the elliptic
  problem}
  \label{simu:ellip}

\subsection{Introduction}
\label{subsec_simu_ellip_intro}

In this section the efficiency of the numerical method
  introduced in sections~\ref{elip} and \ref{num} for the singular perturbation problem
  (\ref{P1}), (\ref{bc1}) is investigated through numerical
  experiments. These experiments are carried out on a two dimensional
  uniform Cartesian mesh. Two sets of test cases are presented. In the first
  one, the anisotropy, or magnetic field, is oblique, which means that
  it is assumed uniform in space, but not necessarily aligned
  with any coordinate axis. In the second set, the field
  direction is non uniform. In both cases, the strength of the anisotropy
  is assumed uniform and is given by the value of $\varepsilon$. An
  analytical solution $\phi_a$ is constructed for the singular perturbation
  problem (\ref{P1}), (\ref{bc1}) and is compared with its approximation
  $\phi^h$ computed on the mesh. 
  For the test cases, the following $L^1$, $L^2$ and $L^\infty$
  norms are used to estimate the errors between the numerical approximation
  $\phi^h$  and the analytical solution 
  $\tilde\phi_a$:
\begin{equation}\label{eq:def:error:norm}
  \begin{split}
e_1 &= \frac{\| \tilde\phi_a - \phi^h \|_{L^1}}{ \|\tilde\phi_a\|_{L^1}} = 
\frac{\sum_{i,j} |  \phi_a(x_i,y_j) - \phi^h(i,j)|}{\sum_{i,j} |
  \phi_a(x_i,y_j) |},
\\
e_2 &= \frac{\|\tilde\phi_a  - \phi^h  \|_{L^2}}{ \|\tilde\phi_a\|_{L^2}} = 
\frac{(\sum_{i,j} |\phi_a(x_i,y_j)  - \phi^h_{i,j}  |^2)^{\frac{1}{2}}}{(\sum_{i,j} |
  \phi_a(x_i,y_j) |^2)^{\frac{1}{2}}}, 
\\   e_{\infty} &= \frac{\| \tilde\phi_a  - \phi^h  \|_{L^\infty}}{
  \|\tilde\phi_a\|_{L^\infty}} = \frac{\max_{i,j} |\phi_a(x_i,y_j)  - \phi^h_{i,j}|}{\max_{i,j} |\phi_a(x_i,y_j)|} \,.
  \end{split}
\end{equation}

\subsection{Numerical results for an oblique magnetic field}

\subsubsection{Introduction and test case settings}

For these numerical experiments the simulation domain is
  the square 
$\Omega = [0,1] \times [0,1]$. The magnetic field is
  defined by $ B = (\sin\alpha , \cos \alpha  , 0      )$,
with $\alpha$ the angle of the b-field with the $x$-axis ranging from
0 to ${\pi}/{2}$. In order to validate the numerical method an 
analytical solution denoted $\phi_a$ for problem (\ref{P1}),
(\ref{bc1}) is constructed. It is written
\begin{equation}\label{eq:analytical:f}
\begin{split}  
  \phi_a(x,y) &= \sin\big(x \sin (\alpha)   - y
  \cos (\alpha) \big) + b \cdot \nabla H(x,y) \,, \\
  f^\varepsilon_a(x,y) &= - b \cdot \nabla \big(\nabla \cdot ((b\otimes b)
  \nabla  H(x,y))\big) \\
  & \hspace*{2cm} + \varepsilon \big( \sin\big(x \sin (\alpha)   - y
  \cos (\alpha) \big) + b \cdot \nabla H(x,y) \big) \,,\\
  H(x,y) &=  \big((x-1)(y-1)xy\big)^3 \,.
\end{split}  
\end{equation}
The function $\phi_a$  
is the solution of problem (\ref{P1}), (\ref{bc1}) with the right-hand side
$f^\varepsilon_a$. $\phi_a$ presents itself as decomposed into $p^\varepsilon$ (first terms) and $q^\varepsilon$ (second term). Note also  that $f^{\varepsilon}_a$ can be decomposed as  $f^{\varepsilon}_a
= f^{(0)}_a + \varepsilon f^{(1)}_a$ with $f^{(0)}_a =- b\cdot \nabla h $ and
$h =\nabla\cdot ( (b \otimes b) \nabla H(x,y))$. The function $h$
verifies homogeneous Dirichlet boundary conditions on the domain boundaries,
which implies, according to theorem \ref{fondamental}, that $f^{(0)}_a \in K^{\perp}$ and the compatibility
condition \eqref{comp} is satisfied. However, for the simulations carried out below, the construction of the
right-hand side $f_a^\varepsilon$ is performed using the discrete
operators $ \left(   b \cdot  \nabla  \right)_{\mbox{\scriptsize{app}}}$ and 
$\nabla \cdot ( \hspace*{2 mm} \cdot \hspace*{2 mm}  b )_{\mbox{\scriptsize{app}}}$
in order to ensure that the compatibility condition \eqref{comp} is satisfied by
the discrete operators, namely $ f_a^{(0)} \in  K_{\mbox{\scriptsize{app}}}^{\perp} $, 
where
\begin{equation*} 
 K_{\mbox{\scriptsize{app}}} = \{  \phi \; / \; \nabla \cdot ( b\phi )_{\mbox{\scriptsize{app}}} = 0    \} \,, \qquad
 K_{\mbox{\scriptsize{app}}}^{\perp} = \left(   b \cdot  \nabla  \right)_{\mbox{\scriptsize{app}}}
(\mathcal{W}_0) \,.
\end{equation*}

\subsubsection{Homogeneous Neumann boundary conditions} \label{oblhom}

This simulation is run with $\alpha = \pi/3$. On figure \ref{erro369}, we represent
the relative 
errors as functions
of the mesh sizes for different values of $\varepsilon$ ranging from $10^{-3}$ to $10^{-9}$. \begin{figure}[!ht]
\begin{center}
\begin{minipage}[c]{0.48\textwidth}
  \subfigure[$\varepsilon=10^{-3}$.\label{erro3}]{%
  \includegraphics[width=\textwidth]{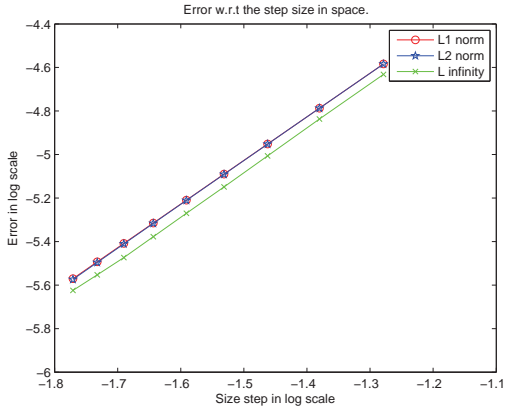}}
\end{minipage}\hfill%
\begin{minipage}[c]{0.48\textwidth}
\subfigure[$\varepsilon=10^{-6}$.\label{erro6}]{%
\includegraphics[width=\textwidth]{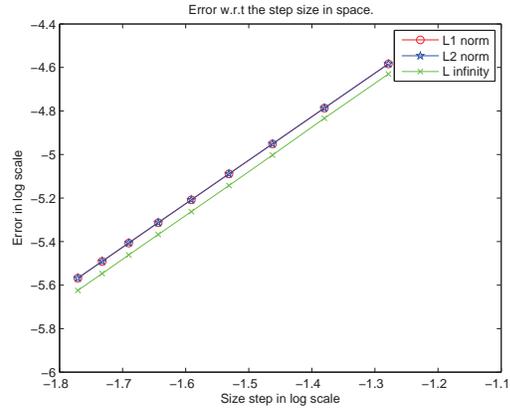}}
\end{minipage} 
\begin{minipage}[c]{0.48\textwidth}
\subfigure[$\varepsilon=10^{-9}$.\label{erro9}]{%
\includegraphics[width=\textwidth]{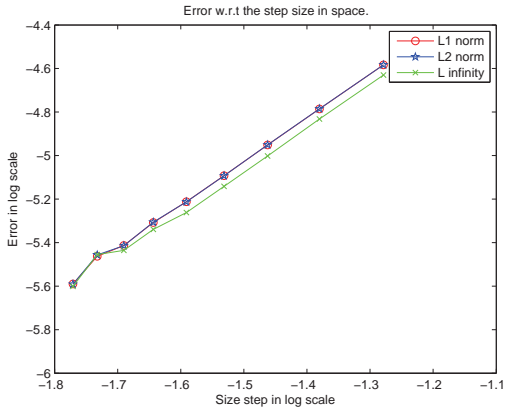}}
\end{minipage}\hfill%
\begin{minipage}[c]{0.45\textwidth}
\caption{Oblique magnetic field test case with $\alpha=\pi/3$:
  error norms, defined by~\eqref{eq:def:error:norm}, for the solution
  $\phi^\varepsilon$ as a function of the mesh size, in decimal
  logathimic scales, and for different
  values of $\varepsilon$.}\label{erro369}
\end{minipage}
\end{center}
\end{figure}
The curves of figure~\ref{erro369} are plotted using logarithmic
decimal scales. We observe a linear decrease of the errors with
vanishing mesh sizes, with a slope equal to 2, which proves that the global scheme is second
order accurate. More importantly, we observe from figures~\ref{erro3}
and \ref{erro6}, that the precision remains the same while 
$\varepsilon$ is decreased by three orders of magnitude.
However, for the more refined grids using the smallest value of
$\varepsilon$ of this simulation set ($10^{-9}$, see figure~\ref{erro9}), a slight degradation of the convergence is observed for small mesh sizes.

This slight degradation can be explained. Indeed,
$p^{\varepsilon}$ is given by a stiff problem, since $\varepsilon
p^\varepsilon$ is obtained as the difference of two quantities
scaling as $\varepsilon^0= {\mathcal O} (1)$ (see (\ref{bfp}), (\ref{fp})). 
To investigate the influence of $\varepsilon$ on the accuracy
of the  approximation of $p^{\varepsilon}$, the $L^\infty$ norm of  the relative error made on $p^\varepsilon$ and on $\nabla
\cdot (b p^\varepsilon$) as functions of $\varepsilon$ are plotted on
figure~\ref{fig:p:study}. 

\begin{figure}[!ht]
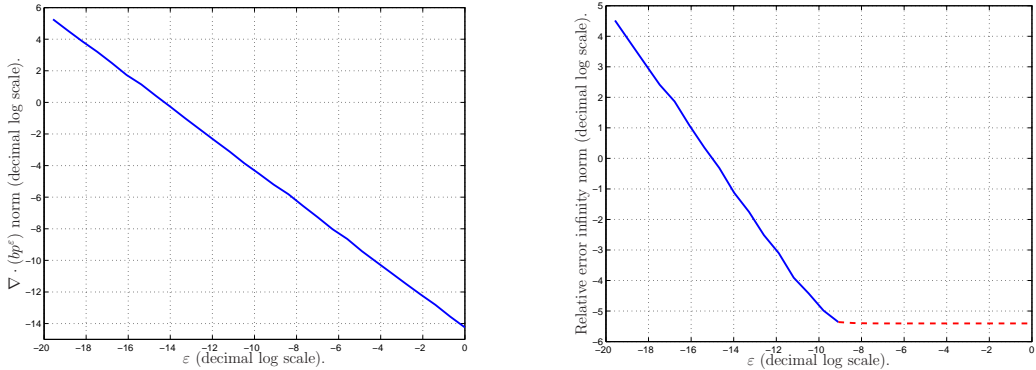

\begin{center}
\begin{minipage}[c]{0.45\textwidth}
\subfigure[Infinity norm for $\nabla \cdot (b p^\varepsilon)$
      as a function of $\varepsilon$ in decimal log. scales.\label{errdivp}   ]{
  \psfrag{X}[][][0.5]{$\varepsilon$ (decimal log
        scale).}
\psfrag{Y}[][][0.5]{$\nabla \cdot (b
  p^\varepsilon)$ norm  (decimal log scale).}
    \includegraphics[width=\textwidth]{divpeps60.epsc}  }
  \end{minipage}\hfill%
  \begin{minipage}[c]{0.45\textwidth}
    \subfigure[Relative error in infinity norm for $p^\varepsilon$
      as a function of $\varepsilon$ in decimal log. scales.\label{errpep}   ]{
        \psfrag{X}[][][0.5]{$\varepsilon$ (decimal log scale).}
      \psfrag{Y}[][][0.5]{Relative error infinity norm (decimal log scale).}
      \includegraphics[width=\textwidth]{errelp60.epsc}  }
  \end{minipage}  
 \begin{minipage}[c]{\textwidth}
   \caption{Oblique magnetic field test case for $\alpha = \pi/3$
     and $\Delta x = \Delta y = 1/60$. Approximation of the
     $p^\varepsilon$ part of the solution.}\label{fig:p:study}
   \end{minipage}
\end{center}
\end{figure}

Figure~\ref{errdivp} shows a linear behavior of 
$\nabla \cdot(b p^\varepsilon)$ with vanishing $\varepsilon$ (in log scale). 
To explain this feature, we note that the discretization of the second order
operator in (\ref{fp}) provides a computation of $\varepsilon \left(\nabla
  \cdot ( \cdot \hspace*{3 mm} b )\right)_{\mbox{\scriptsize{app}}} (p^\varepsilon)$
with the precision of the linear system solver used for
the computation of $g^\varepsilon$, which is limited by round-off errors. This error is amplified after multiplication by the factor $1/\varepsilon$. This analysis still holds for the  accuracy of $p^\varepsilon$  as a function of $\varepsilon$ 
 represented on figure~\ref{errpep}  with slight
differences. For the largest values of $\varepsilon$, we observe a
plateau (red dashed line) explained by the discretization error
of the discrete operators. The space discretization introduced here
is second order accurate, i.e. is $\mathcal{O}(h^2)$ where $h = \max(\Delta x, \Delta y)$.
Since the right-hand side is well prepared 
this error only applies to the $ \varepsilon f^{(1)}$ part of
$f^\varepsilon$ and is then proportional to $\varepsilon \mathcal{O}(h^2)$ in $b\cdot \nabla g^\varepsilon$,
giving rise to a $\mathcal{O}(h^2)$ consistency error for
$p^\varepsilon$. The value of the plateau is thus only dependent
of the mesh sizes and does not depend on the values of $\varepsilon$. With vanishing
values of $\varepsilon$ the round-off errors due to the linear system solver
grow linearly (in log scale) until they reach the consistency error ($\mathcal{O}(h^2)$). This occurs for a value of $\varepsilon$ which, for
this test case, can be estimated as approximately $\varepsilon=10^{-9}$. For smaller
$\varepsilon$, the discretization error is negligible compared to the round-off errors amplified by the factor $1/\varepsilon$ and the
accuracy of $p^\varepsilon$ deteriorates linearly with vanishing $\varepsilon$.

The accuracy of the approximation of $p^\varepsilon$ can be made
  totally independent of $\varepsilon$ under the
  assumption that $f^{(0)} = 0$. In this case, both $b\cdot \nabla
  g^\varepsilon$ and $f^\varepsilon$ scale as $\varepsilon$, 
  providing then an approximation of $p^\varepsilon$ independent of $\varepsilon$.
  The numerical methods introduced in \cite{DDLNN_sub, DDN_MMS10} have been developed under this
  assumption that $f^{(0)}=0$. The present paper is developed under a weaker
  hypothesis, required by the application to the Euler-Lorentz model in
  the drift-limit. This explains why a comparable accuracy cannot be reached.
Therefore, strictly speaking, our scheme is AP for the computation of $p^{\varepsilon}$ only when
  $f^{(0)}=0$, or, when $f^{(0)}\not=0$, only if the round-off errors brought by the linear system solver are smaller than the discretization error. Still, it is AP without any restriction for the computation of $q^{\varepsilon}$ (i.e. even when $f^{(0)}\not=0$).

The next simulation is aimed at investigating whether the accuracy depends on the angle between $b$ and the coordinate axes. For this purpose, simulations are
carried out on a mesh composed of $40\times 40$ cells and for  $\alpha$ ranging form 0 to $\pi/2$. When $\alpha=0$ the $b$
field is aligned with the $x$-axis and when $\alpha=\pi/2$, it is aligned with the $y$-axis. The relative errors are
displayed as functions of $\alpha$ on figure~\ref{errtot40}.
\begin{figure}[!ht]
\begin{center}
\includegraphics[width=0.65\textwidth,height=0.3\textheight]{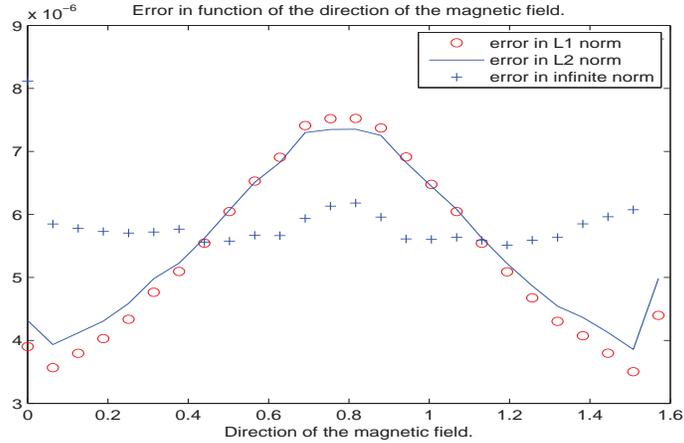} 
\end{center}
\caption{Oblique magnetic field test case for $\varepsilon=10^{-9}$
  and $\Delta x = \Delta y = 1/40$. Norms of the relative error~\eqref{eq:def:error:norm} as a function of
  the angle of the magnetic field with
  the $x$-axis $\alpha$.}\label{errtot40}
\end{figure}
We observe that the variations of the errors are small on the whole range of
angles. This confirms that the method provides accurate results, even when the mesh is far from consistent with the $b$-field direction.

\subsubsection{Inhomogeneous Neumann boundary conditions} 
\label{oblinhom}

We remark that $\phi^{\varepsilon} (x,y) = 2
x^2+y^2$ is an analytical solution of system (\ref{sysr}), (\ref{bc}) for $f_2(x,y) = \varepsilon
(2x^2+y^2)$ and $\kappa = - \nabla \cdot ( b f)$. For this analytical solution and $\varepsilon = 10^{-9}$, we have checked that the relative error does  not exceed $10^{-13}$.

\subsection{Numerical results for a non uniform magnetic field}

\subsubsection{Introduction and test case settings}

In this subsection $\Omega = ]1,2[ \times ]1,2[$ and the magnetic field is given by: 
\begin{eqnarray} \label{B}
 B = |B| \,  b , \hspace*{3 mm}  b = \left( \sin (\theta ) , - \cos (
  \theta ) \right) , \hspace*{3 mm} \tan (\theta ) = \frac{y}{x} .
\end{eqnarray} 
For this case, an analytical
solution of (\ref{P1}), (\ref{bc1}) can be found.
We consider $H_{var}$ defined on $[1,2]\times [1,2]$ by $H_{var}(x,y) =
(1-x)^3(1-y)^3(2-x)^3(2-y)^3$. According to Theorem \ref{fondamental}, 
$b\cdot \nabla H_{var} \in K ^{\perp}$  . So $\phi = 1+b\cdot \nabla H_{var}$ is
the solution of (\ref{P1}), (\ref{bc1}) when the right-hand $f^{\varepsilon}$ of
(\ref{P1}) 
has the expression
\begin{eqnarray*} 
f^{\varepsilon} = -b\cdot \nabla \left( \nabla \cdot (b \otimes b) \nabla
H_{var}  \right)
+ \varepsilon \left( 1+b\cdot \nabla H_{var} \right).
\end{eqnarray*}

\subsubsection{Homogeneous Neumann boundary conditions}

On figures \ref{err3var}, \ref{err6var} and \ref{err9var}, we have represented
 the relative 
errors as functions
of the mesh size when $\varepsilon$ goes from $10^{-3}$ to $10^{-9}$. We
observe that all the three norms decrease when the mesh sizes
decrease, in a similar fashion as in the oblique uniform $b$-field. 

 \begin{figure}[!ht]
\begin{center}
\begin{minipage}[c]{0.45\textwidth}
\subfigure[$\varepsilon = 10^{-3}$ \label{err3var}]{
\includegraphics[width=\textwidth]{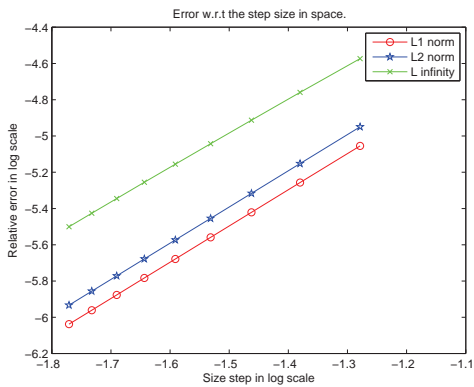}}
\end{minipage}\hfill%
\begin{minipage}[c]{0.45\textwidth}
\subfigure[$\varepsilon = 10^{-6}$ \label{err6var}]{
\includegraphics[width=\textwidth]{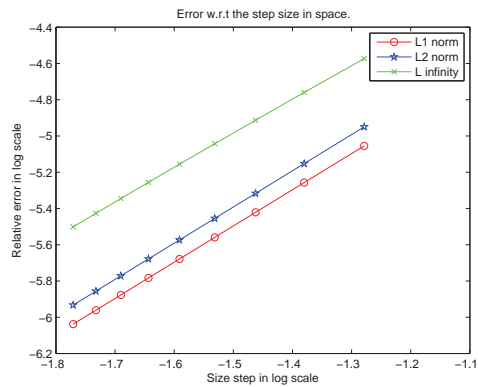}}
\end{minipage} 
\begin{minipage}[c]{0.45\textwidth}
\subfigure[$\varepsilon = 10^{-9}$ \label{err9var}]{
\includegraphics[width=\textwidth]{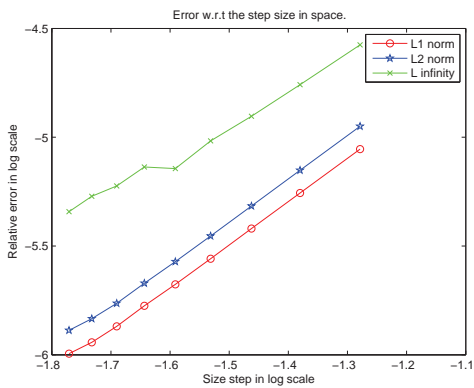}} 
\end{minipage} \hfill
\begin{minipage}[c]{0.45\textwidth}
\caption{Non uniform magnetic field test case:
  error norms, defined by~\eqref{eq:def:error:norm}, for the solution
  $\phi^\varepsilon$ as a function of the mesh size, in decimal
  logathimic scales, and for different values of $\varepsilon$.}\label{errovar369}
\end{minipage}
\end{center}
\end{figure}

\subsubsection{Inhomogeneous Neumann boundary conditions}

We take the test case of subsubsection \ref{oblinhom} again, and we find a similar conclusion: with  $\varepsilon = 10^{-9}$, the relative error in $L^\infty$ norm does not exceed $10^{-11}$.

\setcounter{equation}{0}
\section{Numerical results for the Euler-Lorentz system in the drift
  limit}\label{sec:comp}

\subsection{Introduction and test case settings}

This part is devoted to the validation of the AP-scheme (\ref{DA5e1}),
(\ref{DA5e2}), (\ref{DA5e3}) for the Euler-Lorentz system. Due to the lack of analytical solutions, the validation procedure consists
in comparisons of the AP-scheme with the classical
discretization~\eqref{class2}. 
The classical discretization is subject to a CFL stability condition
that imposes the time step to resolve (i.e. to be smaller than) the fastest time scales involved in the
system. These {\it time-resolved} simulations require a time
step which scales like $\sqrt{\varepsilon}$ (because the CFL condition involves the acoustic wave speed which scales like $1/\sqrt{\varepsilon}$). The 
AP-scheme is designed to be stable independently of $\varepsilon$ when $\varepsilon \to 0$. In these situations, the time step cannot resolve the fastest time scales involved in the system, which leads to {\it under-resolved} simulations. The stability of the AP-scheme in under-resolved situations has be demonstrated in \cite{DDSV_JCP09}. In this case, the requested CFL condition only involves the fluid velocity, which is an $O(1)$ quantity, and not the acoustic speed \cite{DDSV_JCP09} and explains the possibility of using large time steps, independent of $\varepsilon$. We want to check this feature again when the scheme is equipped with our new elliptic solver.

Two test cases are presented, one for an oblique uniform magnetic
field, another one for a non uniform magnetic field with the same expressions as in section~\ref{simu:ellip}. 
 In both cases, the electric field is chosen as $E = (0,0,B_x+
B_y)$, where $B_x$ and $B_y $ are the components of the magnetic
field. The initial condition is defined by the following uniform data:
$n=1$, $(nu)_x=1$, $(nu)_y=-1$ and $(nu)_z=0$ which defines a stationary
solution of the Euler-Lorentz system. A local perturbation of order
$\varepsilon$  in then applied to this stationary state and the evolution of
the system is observed for both the AP and the classical schemes.

\subsection{Numerical results for an oblique uniform magnetic field}

The results for the AP and the classical schemes are compared on
figure~\ref{fig:EL:oblique:resolved} in a resolved case. 
\begin{figure}[!ht]
\begin{center}
\begin{minipage}[c]{0.45\textwidth}
  \subfigure[$n$ (AP-scheme).\label{rhoapres}]{
\includegraphics[width=\textwidth]{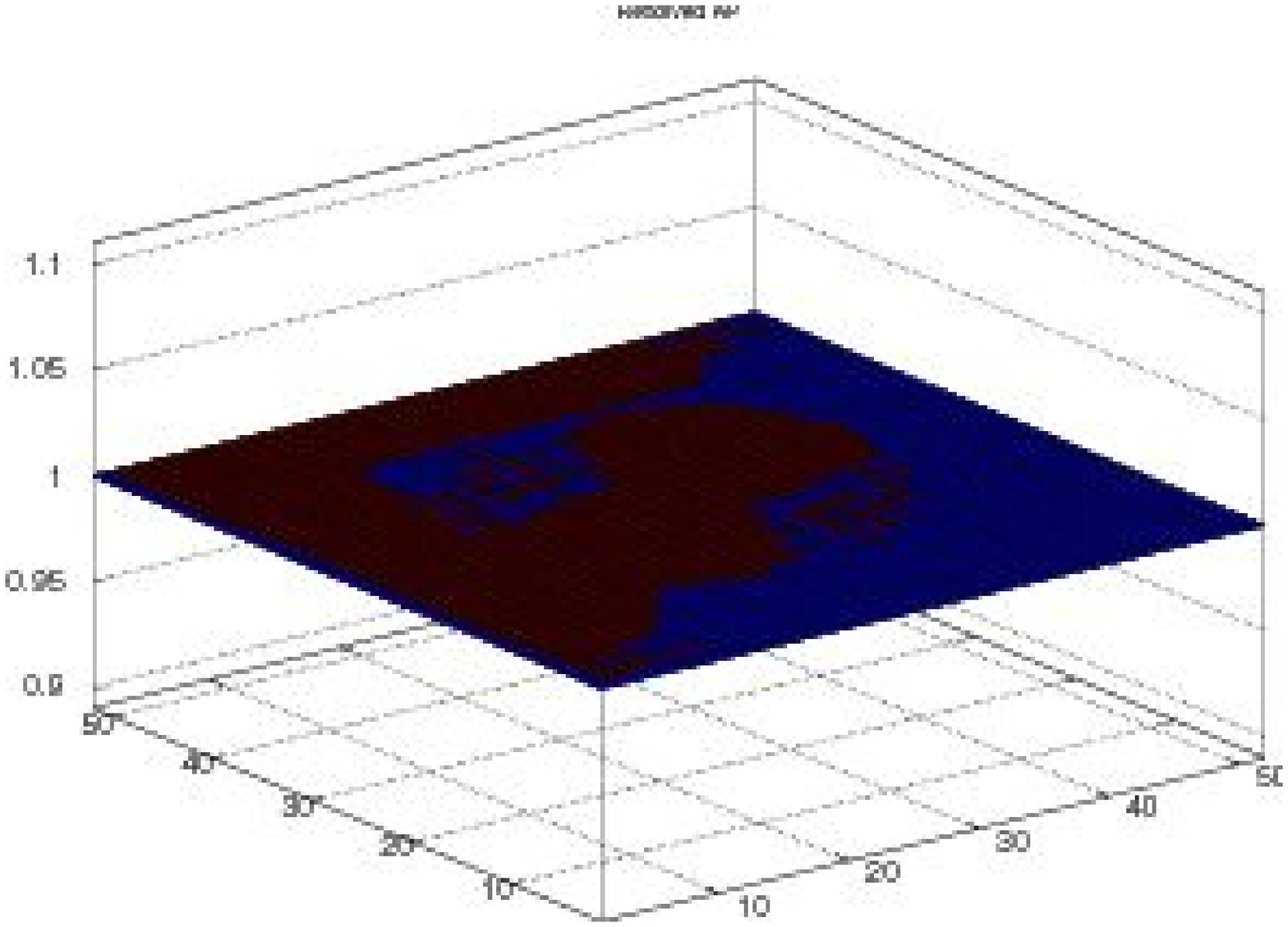}}
\end{minipage}\hfill%
\begin{minipage}[c]{0.45\textwidth}
  \subfigure[$n$ (classical scheme).\label{rhoclassres}]{
\includegraphics[width=\textwidth]{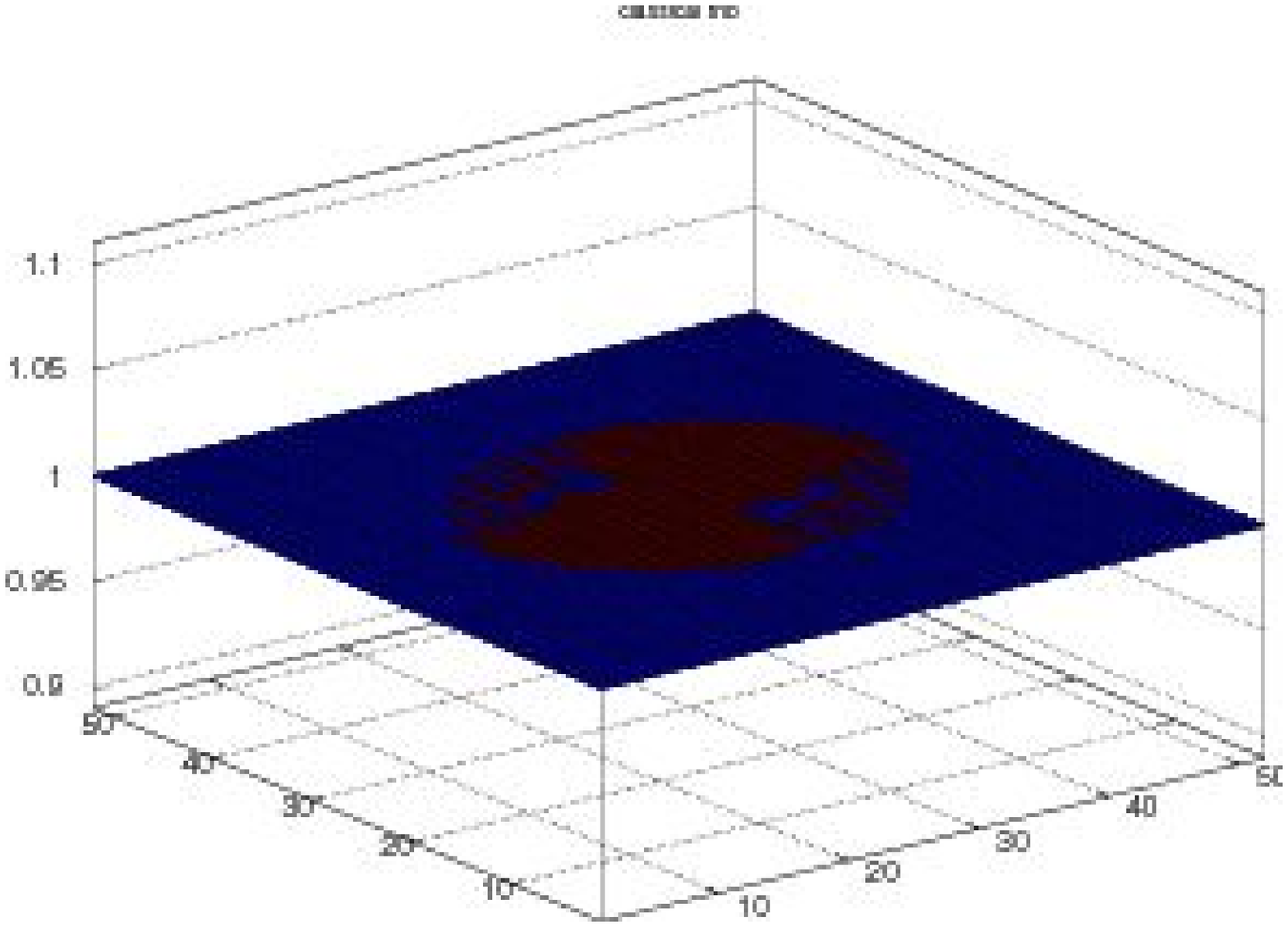}}
\end{minipage} 

\begin{minipage}[c]{0.45\textwidth}
\subfigure[$nu_x$ (AP-scheme).\label{uxapres}]{
\includegraphics[width=\textwidth]{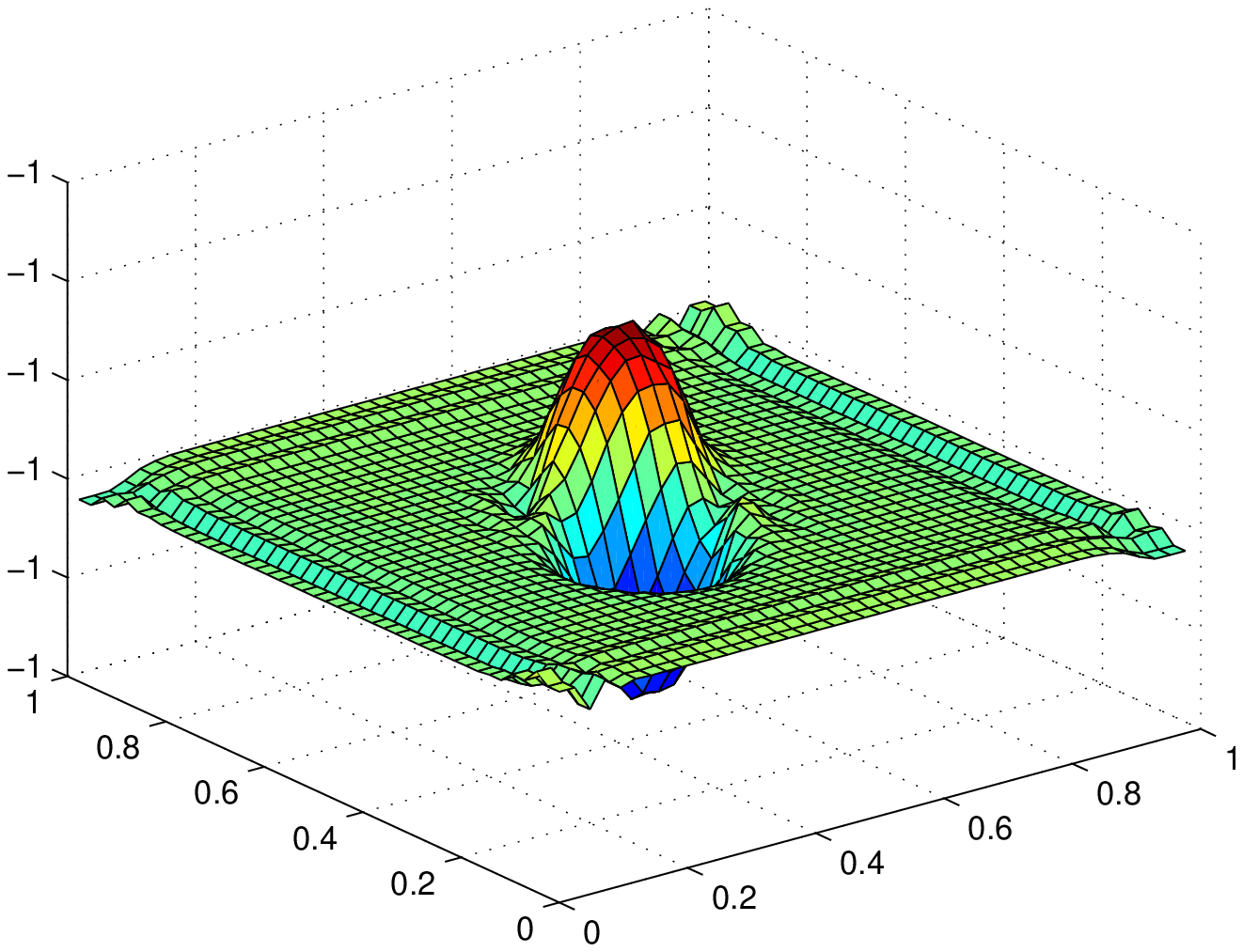}}
\end{minipage}\hfill%
\begin{minipage}[c]{0.45\textwidth}
\subfigure[$nu_x$ (classical scheme).\label{uxclassres}]{
\includegraphics[width=\textwidth]{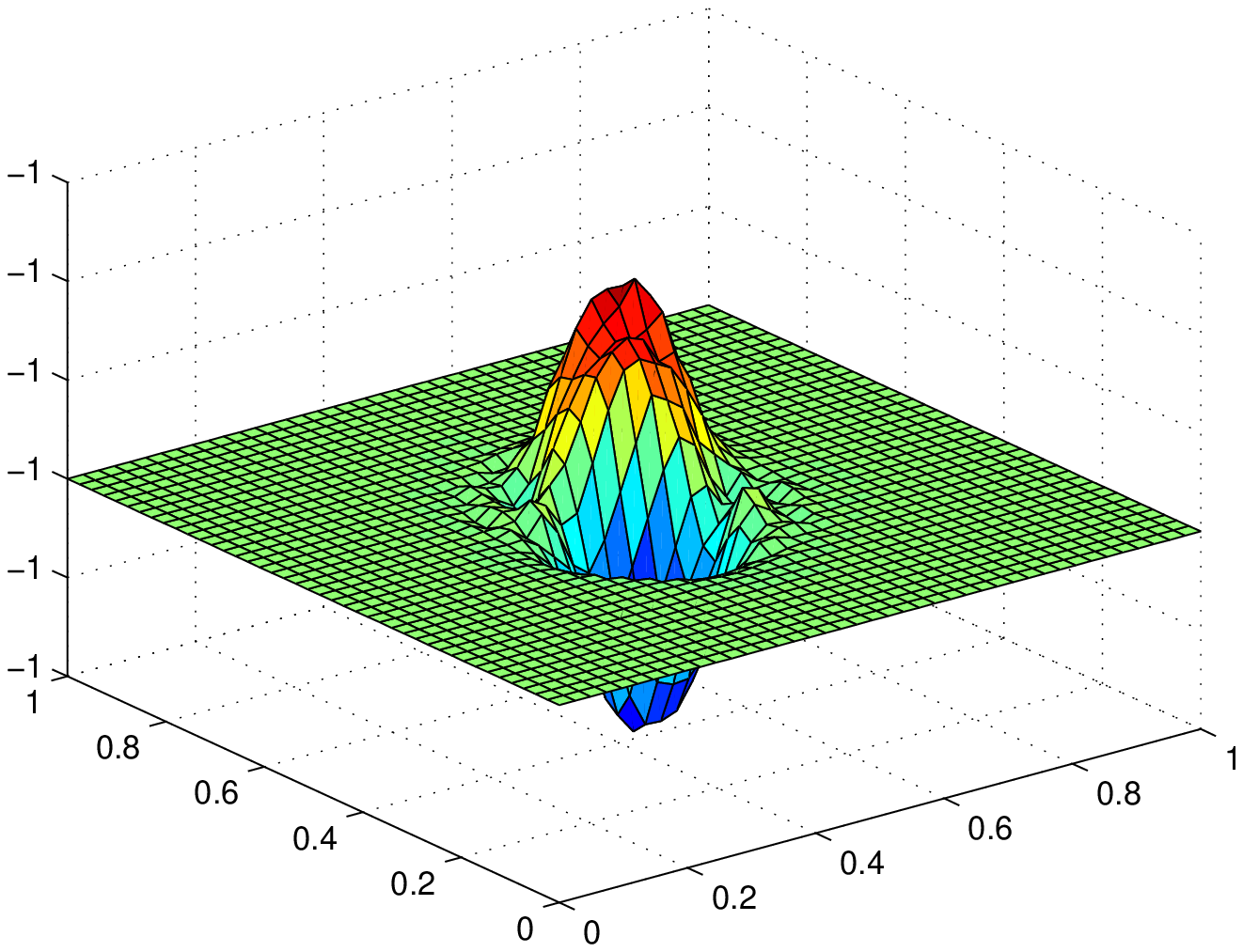}}
\end{minipage} 

\begin{minipage}[c]{0.45\textwidth}
\subfigure[$nu_y$ (AP-scheme).\label{uyapres}]{
\includegraphics[width=\textwidth]{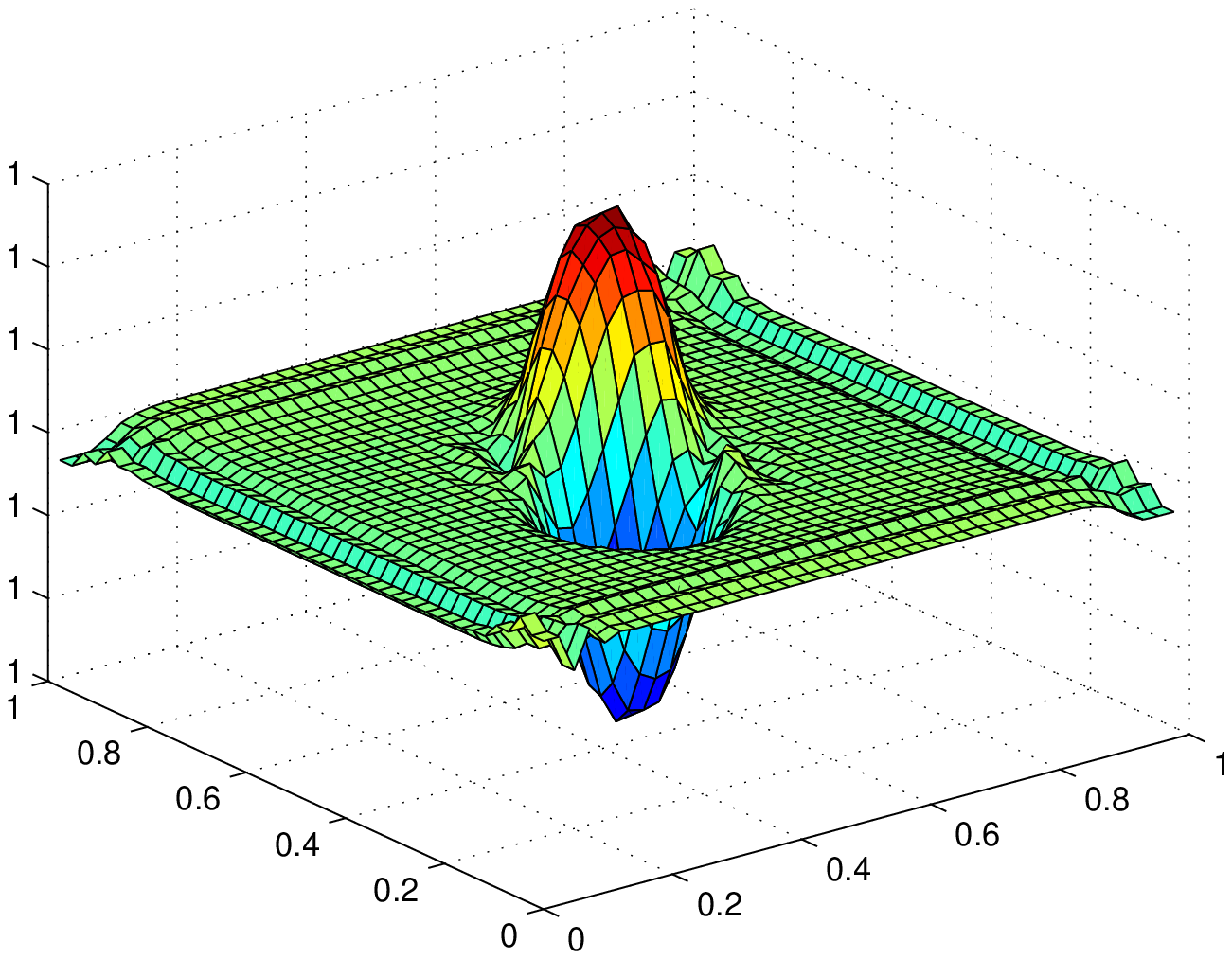}}
\end{minipage}\hfill%
\begin{minipage}[c]{0.45\textwidth}
\subfigure[$nu_y$ (classical scheme).\label{uyclassres}]{
\includegraphics[width=\textwidth]{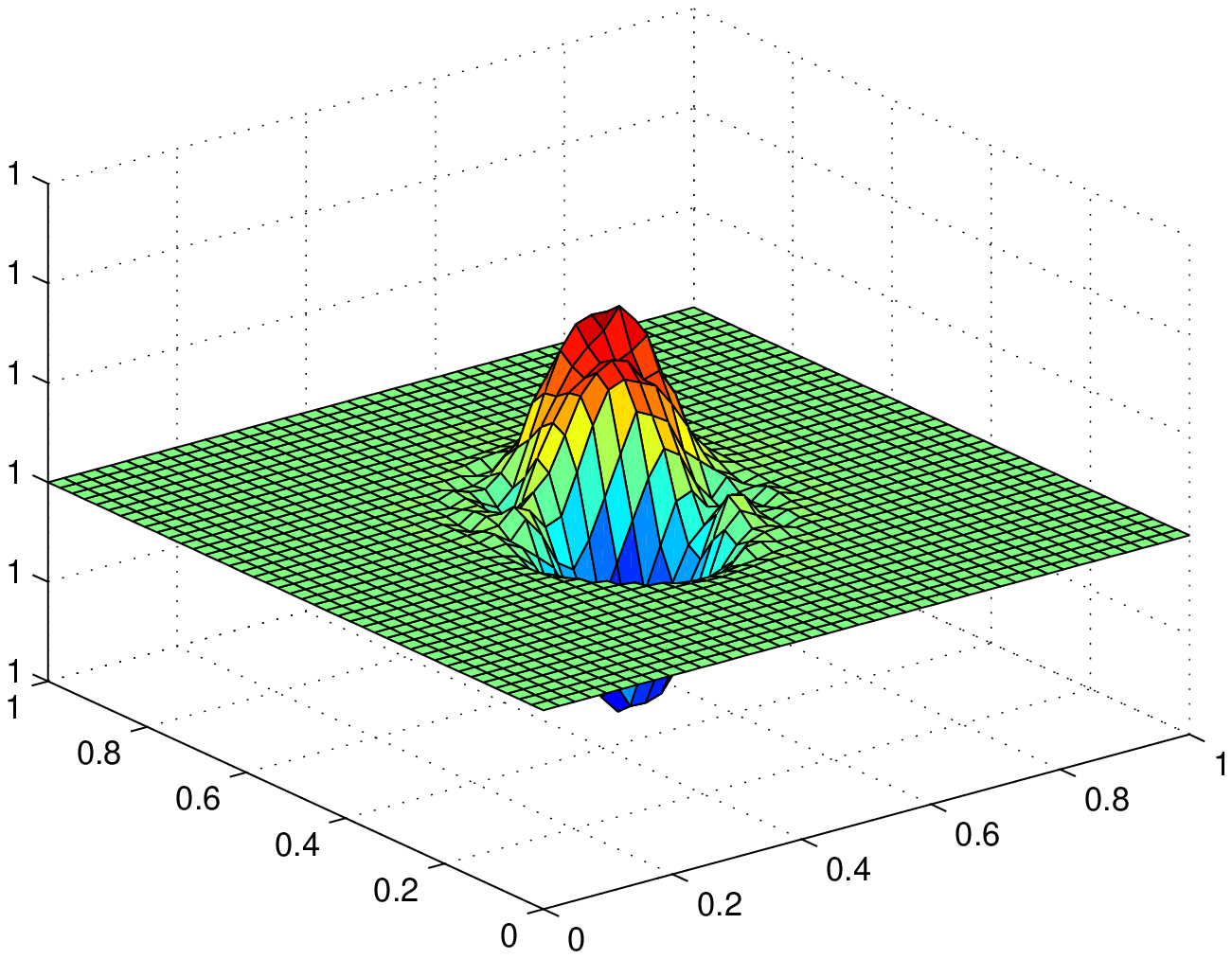}}
\end{minipage} 
\caption{Euler-Lorentz test case for an oblique magnetic field in the
  resolved case at time $t=3.95 \; 10^{-6}$ s:
  density ($n$) and momentum ($nu_x$, $nu_y$) computed by the
  AP-scheme (left) and the classical scheme (right) for $\varepsilon =
  10^{-9}$ and $\Delta x = \Delta y = 1/40$. The angle of the magnetic
field with the $x$-axis is $\alpha=\pi/3$.}\label{fig:EL:oblique:resolved}
\end{center}
\end{figure}
Both schemes provide comparable results. However we observe the
formation of a thin boundary layer on the domain frontiers for the AP-scheme but it is not
responsible for the development of an instability. 

Next we consider the same test case with an under-resolved time step
$\Delta t$ which is 10 times larger than the time step provided by the CFL condition of the classical scheme. These simulation results are collected on
figure~\ref{fig:EL:oblique:unresolved}.%
\begin{figure}[!ht]
\begin{center}
\begin{minipage}[c]{0.45\textwidth}
\subfigure[$n$ (AP-scheme).\label{rhoapnonres}]{
\includegraphics[width=\textwidth]{rhoapnonrescemracs.epsc}}
\end{minipage}\hfill%
\begin{minipage}[c]{0.45\textwidth}
\subfigure[$n$ (classical-scheme).\label{rhoclassnonres}]{
\includegraphics[width=\textwidth]{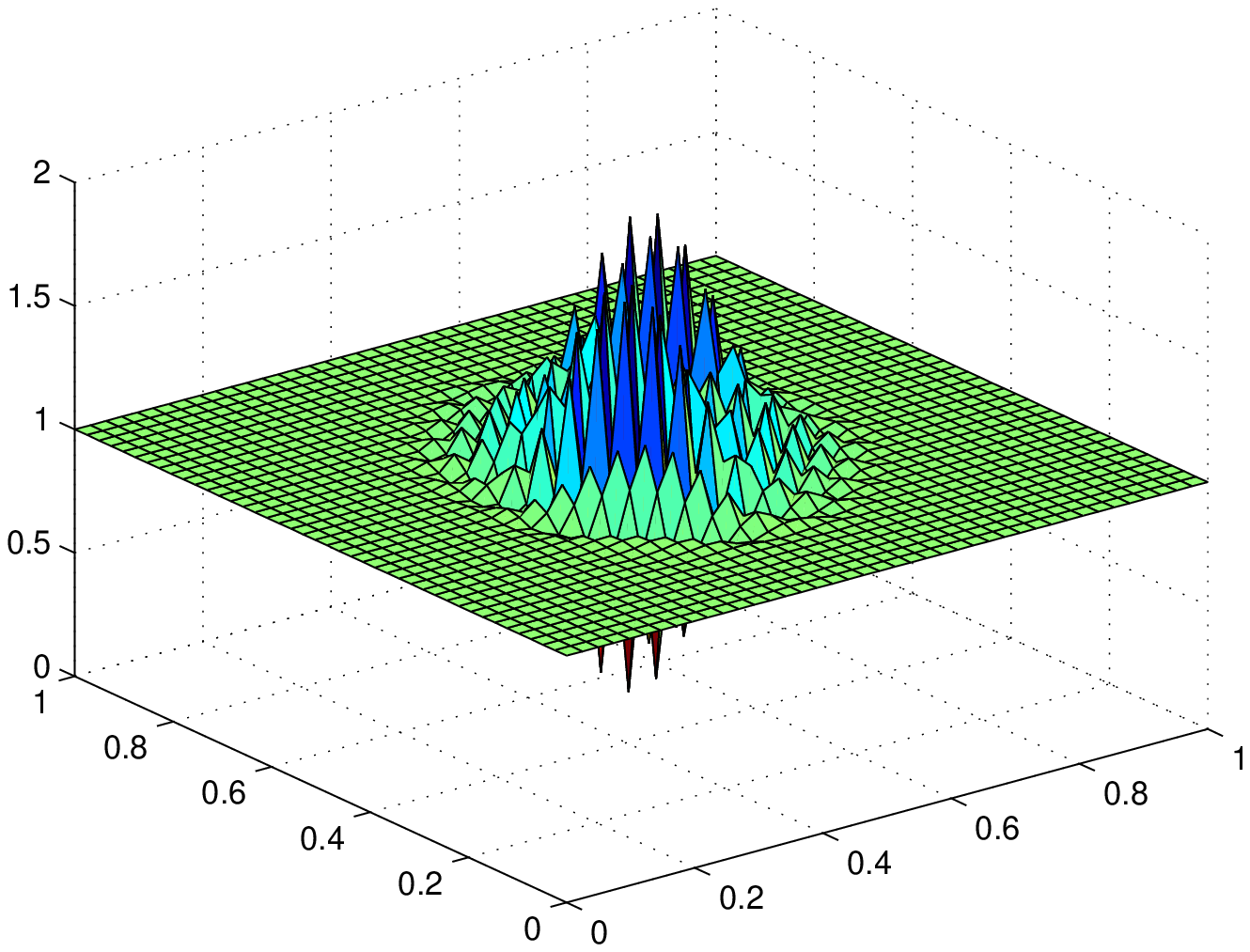}}
\end{minipage} 
\begin{minipage}[c]{0.45\textwidth}
\subfigure[$nu_x$ (AP-scheme)\label{uxapnonres}]{
\includegraphics[width=\textwidth]{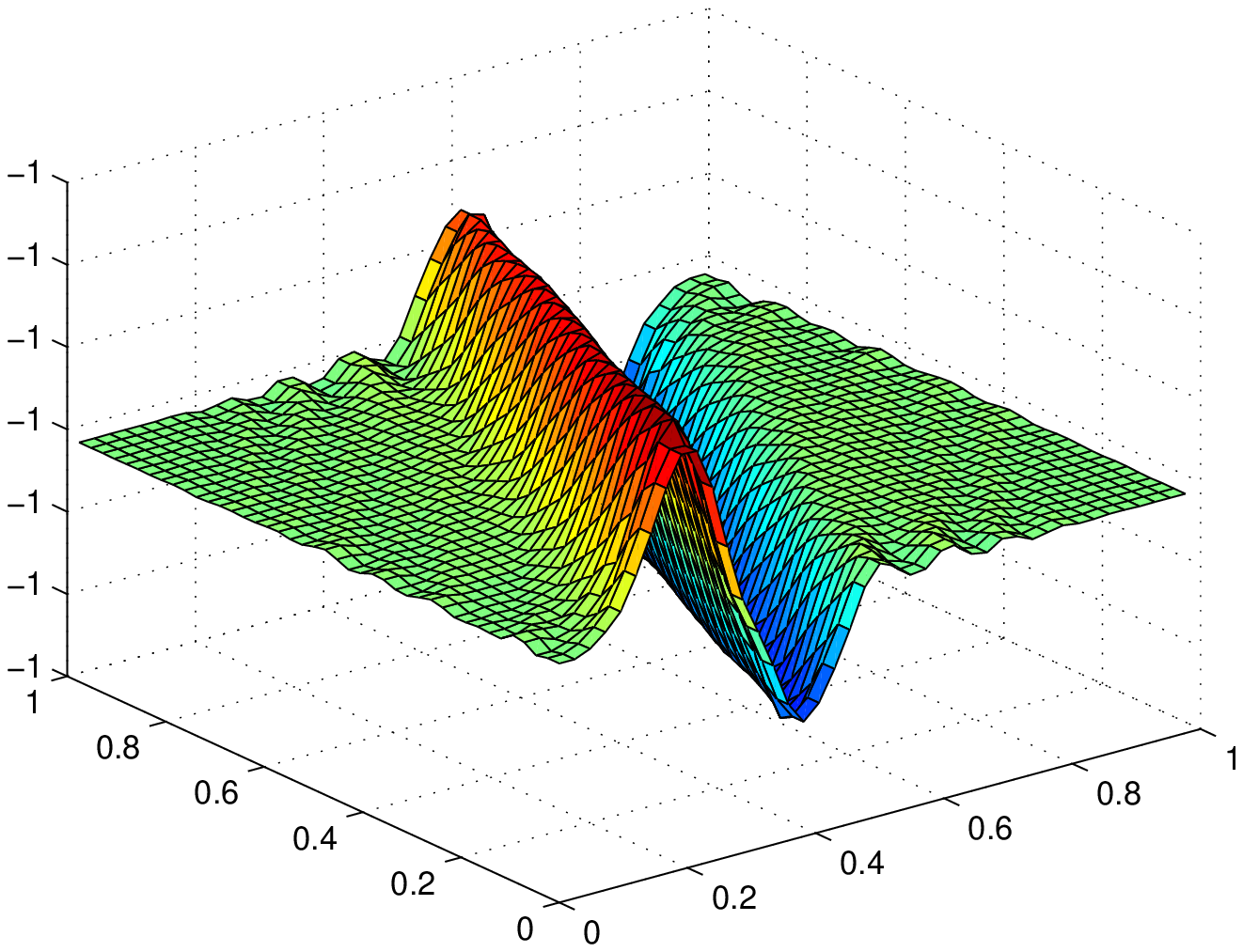}}
\end{minipage}\hfill%
\begin{minipage}[c]{0.45\textwidth}
\subfigure[$nu_x$ (classical-scheme)\label{uxclassnonres}]{
\includegraphics[width=\textwidth]{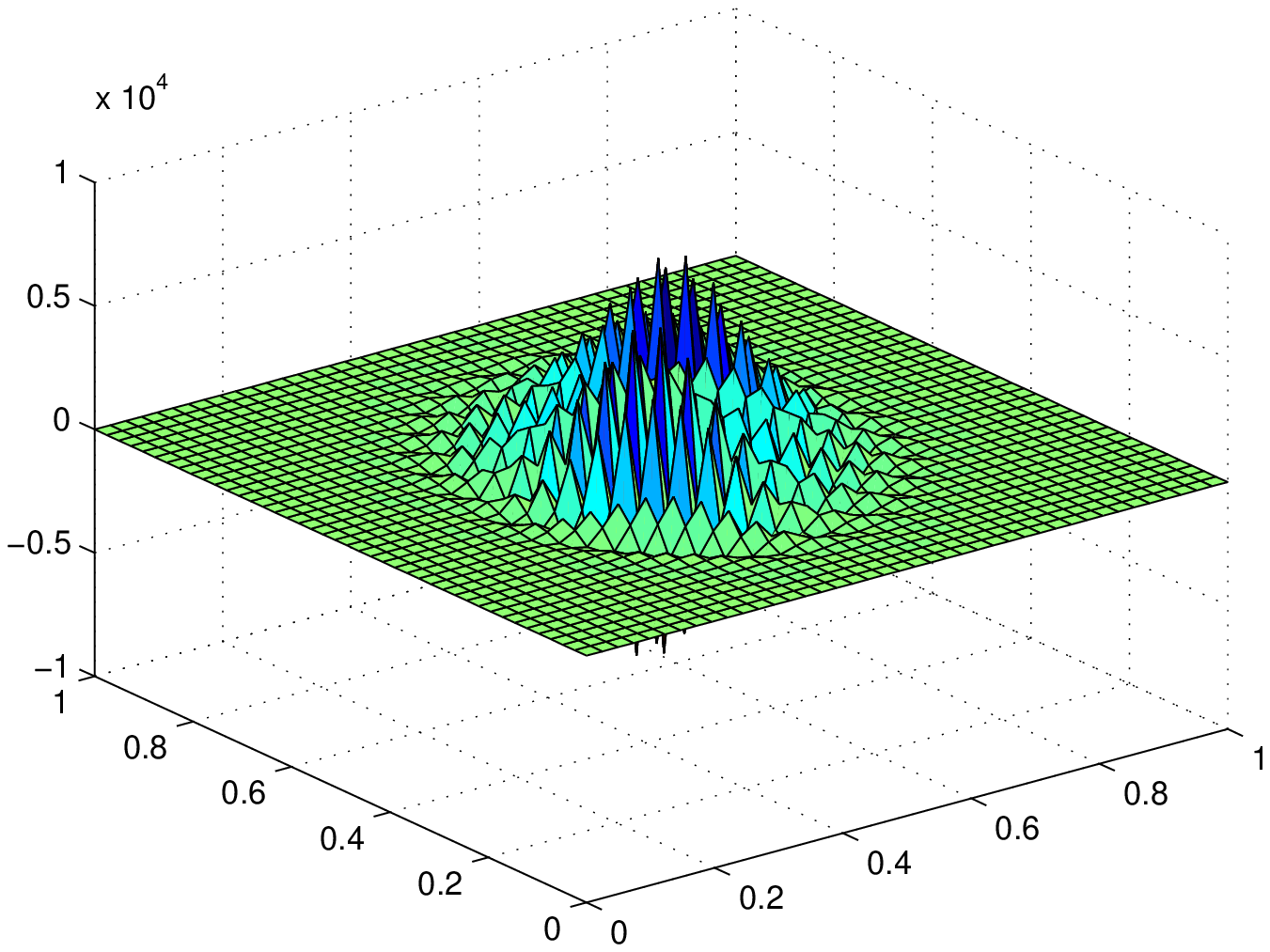}}
\end{minipage} 
\begin{minipage}[c]{0.45\textwidth}
\subfigure[$nu_y$ (AP-scheme) \label{uyapnonres}]{
\includegraphics[width=\textwidth]{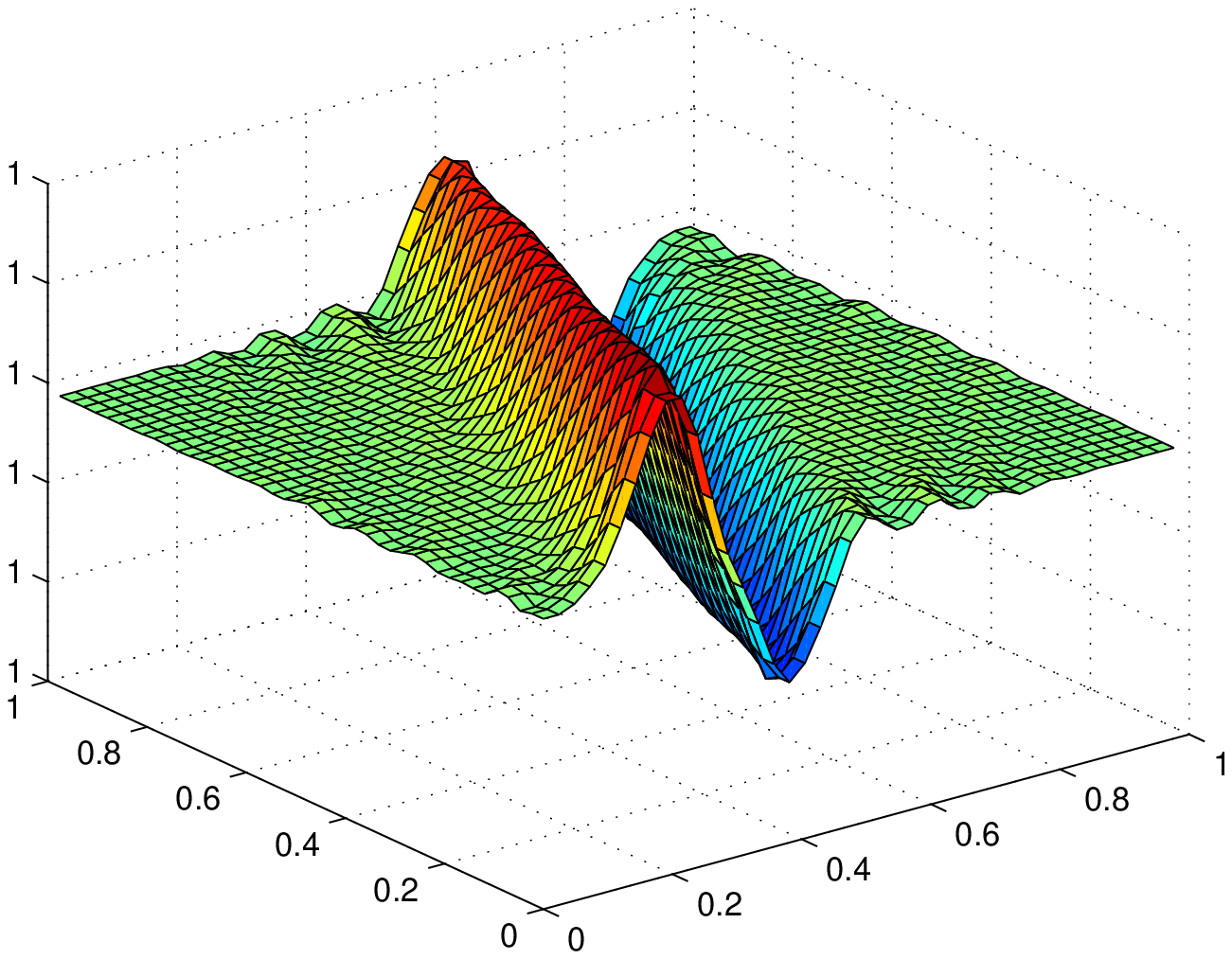}}
\end{minipage}\hfill%
\begin{minipage}[c]{0.45\textwidth}
\subfigure[$nu_y$ (classical-scheme) \label{uyclassnonres}]{
\includegraphics[width=\textwidth]{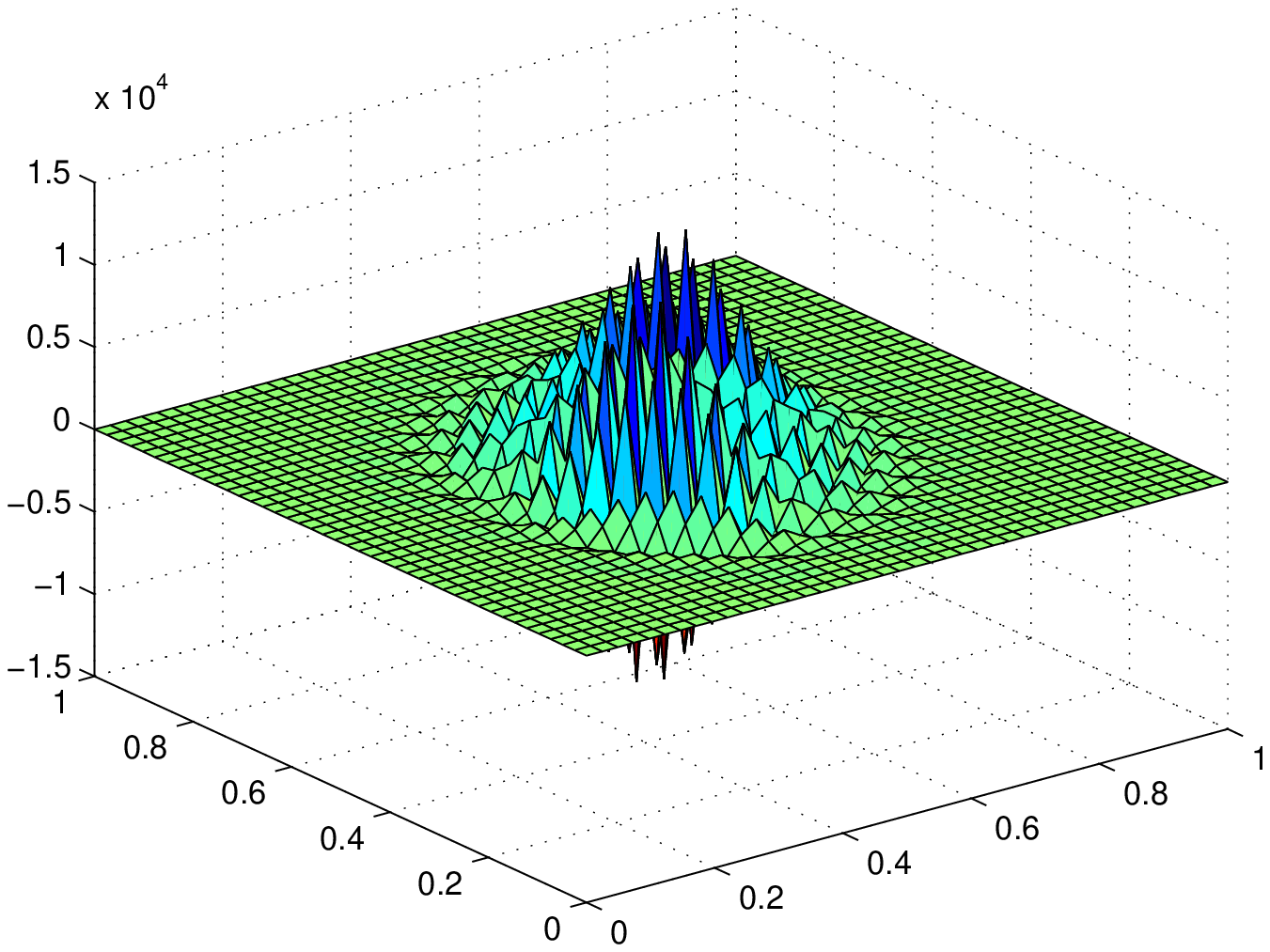}}
\end{minipage} 
\caption{Euler-Lorentz test case for an oblique magnetic field in the
   under-resolved case at time $3.95 \, 10^{-5}$ s:
  density ($n$) and momentum ($nu_x$, $nu_y$) computed by the
  AP-scheme (left) and the classical scheme (right) for $\varepsilon =
  10^{-9}$ and $\Delta x = \Delta y = 1/40$. The angle of the magnetic
field with the $x$-axis is $\alpha=\pi/3$.}
\label{fig:EL:oblique:unresolved}
\end{center}
\end{figure}
In this case, the conventional
scheme leads to unstable results contrary to the AP scheme and
proves the capability of the AP-scheme to provide stable computations
for time steps that resolve neither the acoustic wave-speed nor the
gyration period.

\subsection{Numerical results for a non uniform magnetic field}
\label{subsec_num_nonuniform}

For the non uniform case, $n=1$, $(nu)_x=1$, $(nu)_y=-1$ and
$(nu)_z=0$ are not stationary
solutions to the Euler-Lorentz system.
In particular, with the chosen initial condition, sharp boundary layers are generated. 
 But the the AP scheme can still
be compared with the classical scheme in the resolved case for a
validation procedure. Then we take the same initial conditions as for
the oblique magnetic field case. Figures (\ref{rhoapresvar},
\ref{rhoclassresvar}, \ref{uxapresvar}, \ref{uxclassresvar},
\ref{uyapresvar}, \ref{uyclassresvar}) show that the two schemes provide similar results. 

Next we
consider the under-resolved time step $10 \Delta t$. In this situation  
Fig. \ref{rhoclassnonresvar}, \ref{uxclassnonresvar},
\ref{uyclassnonresvar} show that the classical scheme is unstable. By contrast, Fig.
\ref{rhoapnonresvar}, \ref{uxapnonresvar}, \ref{uyapnonresvar} demonstrate that the AP-scheme provides stable results. The increased numerical diffusion generated by the large time step gives rise to a widening of the boundary layer. Keeping the boundary layer accurate would require some mesh refinment in the vicinity of the boundary. This point is deferred to future work. 
\begin{figure}[!ht]
\begin{center}
\begin{minipage}[c]{0.45\textwidth}
\subfigure[$n$ (classical-scheme) \label{rhoclassresvar}]{
\includegraphics[width=\textwidth]{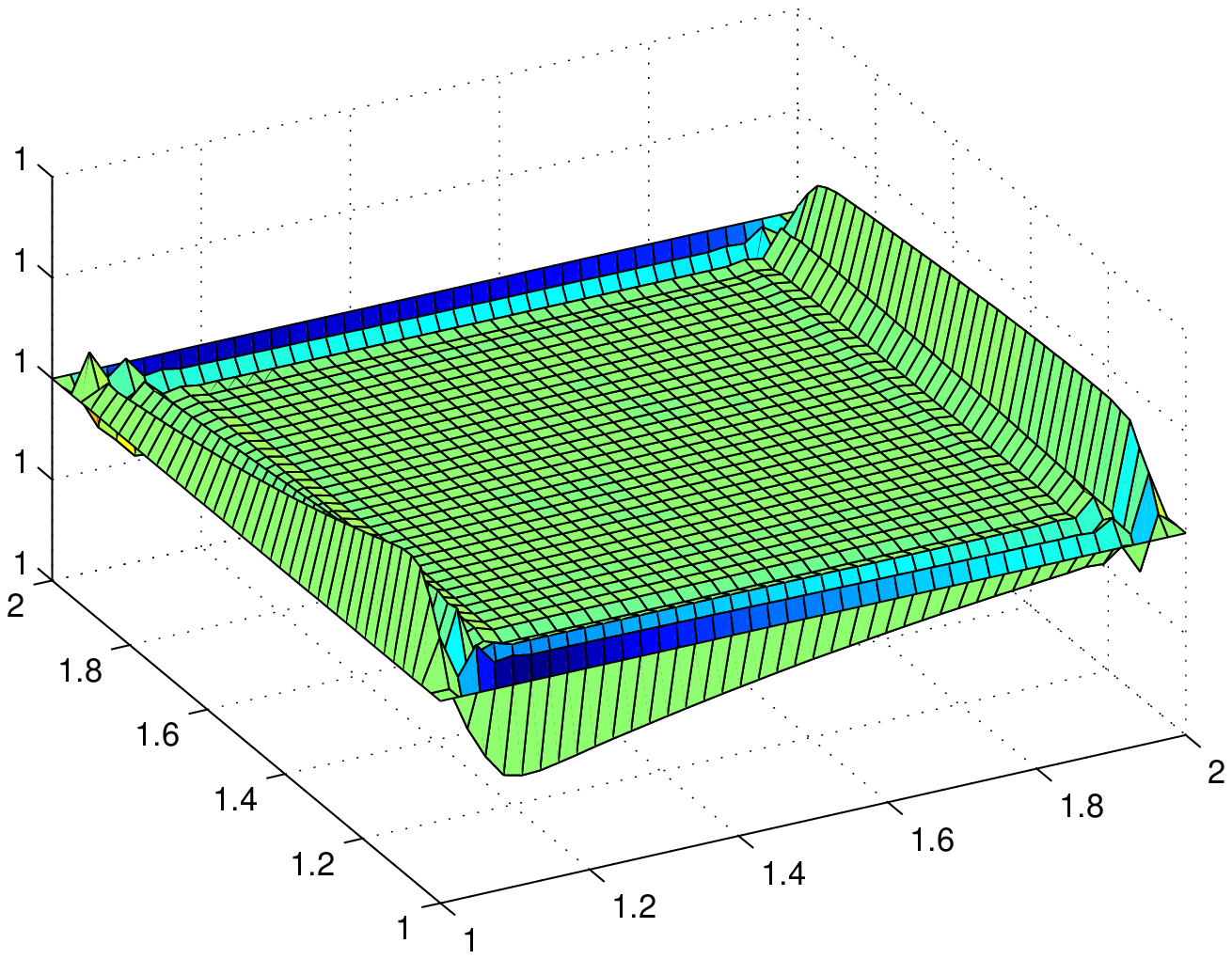}}
\end{minipage}\hfill%
\begin{minipage}[c]{0.45\textwidth}
\subfigure[$n$ (AP-scheme) \label{rhoapresvar}]{
\includegraphics[width=\textwidth]{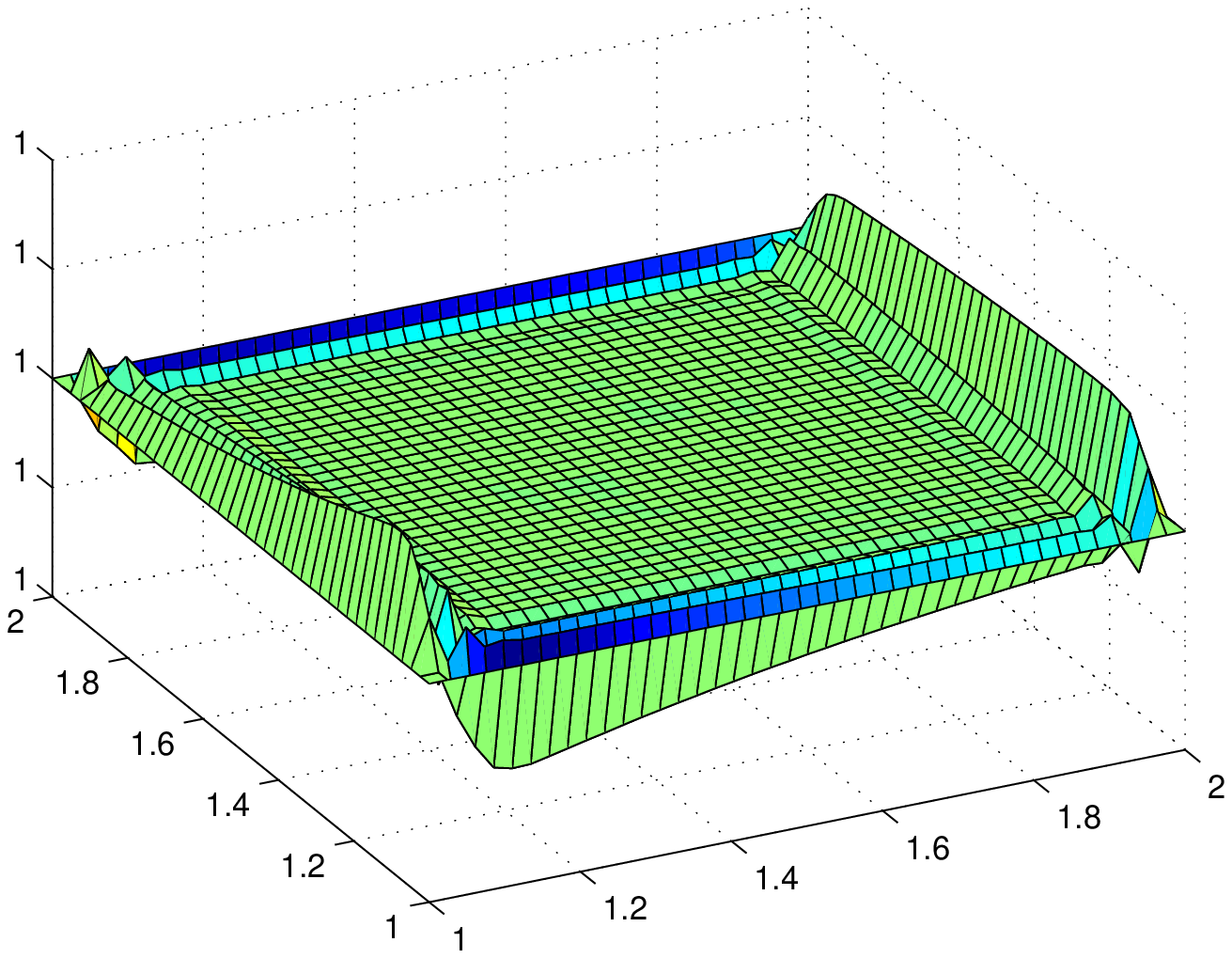}}
\end{minipage} 
\begin{minipage}[c]{0.45\textwidth}
\subfigure[$nu_x$ (classical-scheme) \label{uxclassresvar}]{
\includegraphics[width=\textwidth]{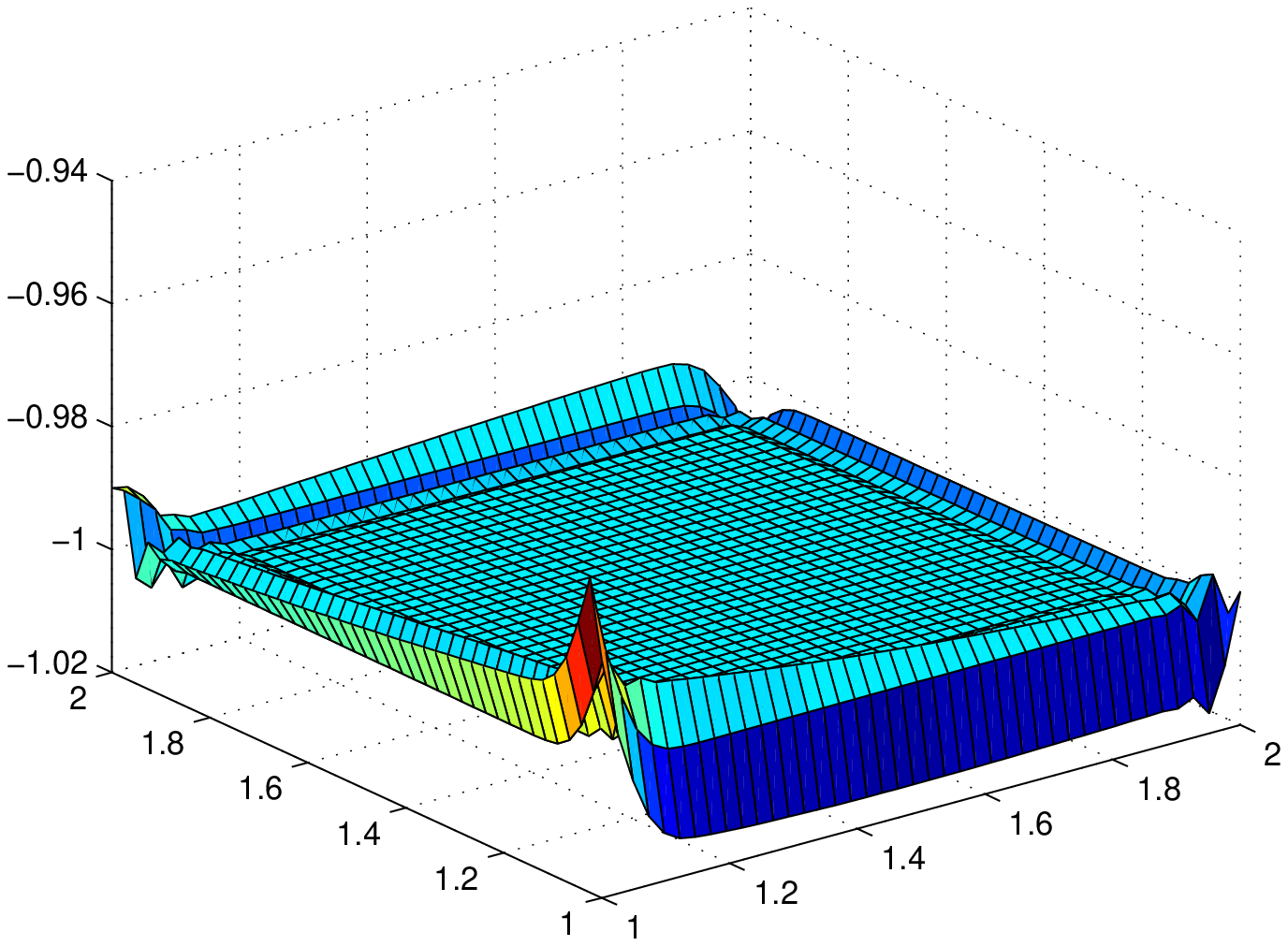}}
\end{minipage}\hfill%
\begin{minipage}[c]{0.45\textwidth}
\subfigure[$nu_x$ (AP-scheme) \label{uxapresvar}]{
\includegraphics[width=\textwidth]{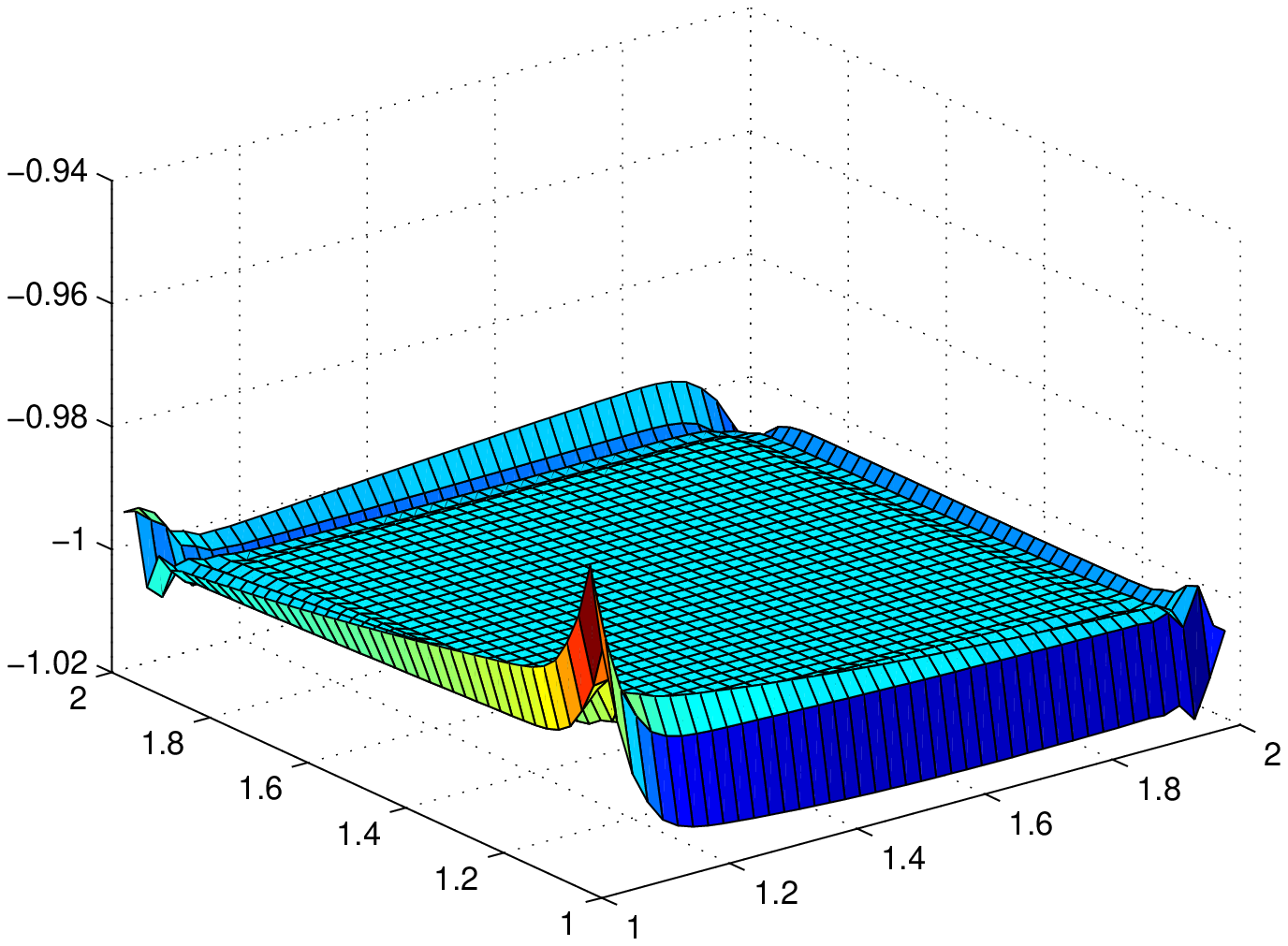}}
\end{minipage} 
\begin{minipage}[c]{0.45\textwidth}
\subfigure[$nu_y$ (classical-scheme) \label{uyclassresvar}]{
\includegraphics[width=\textwidth]{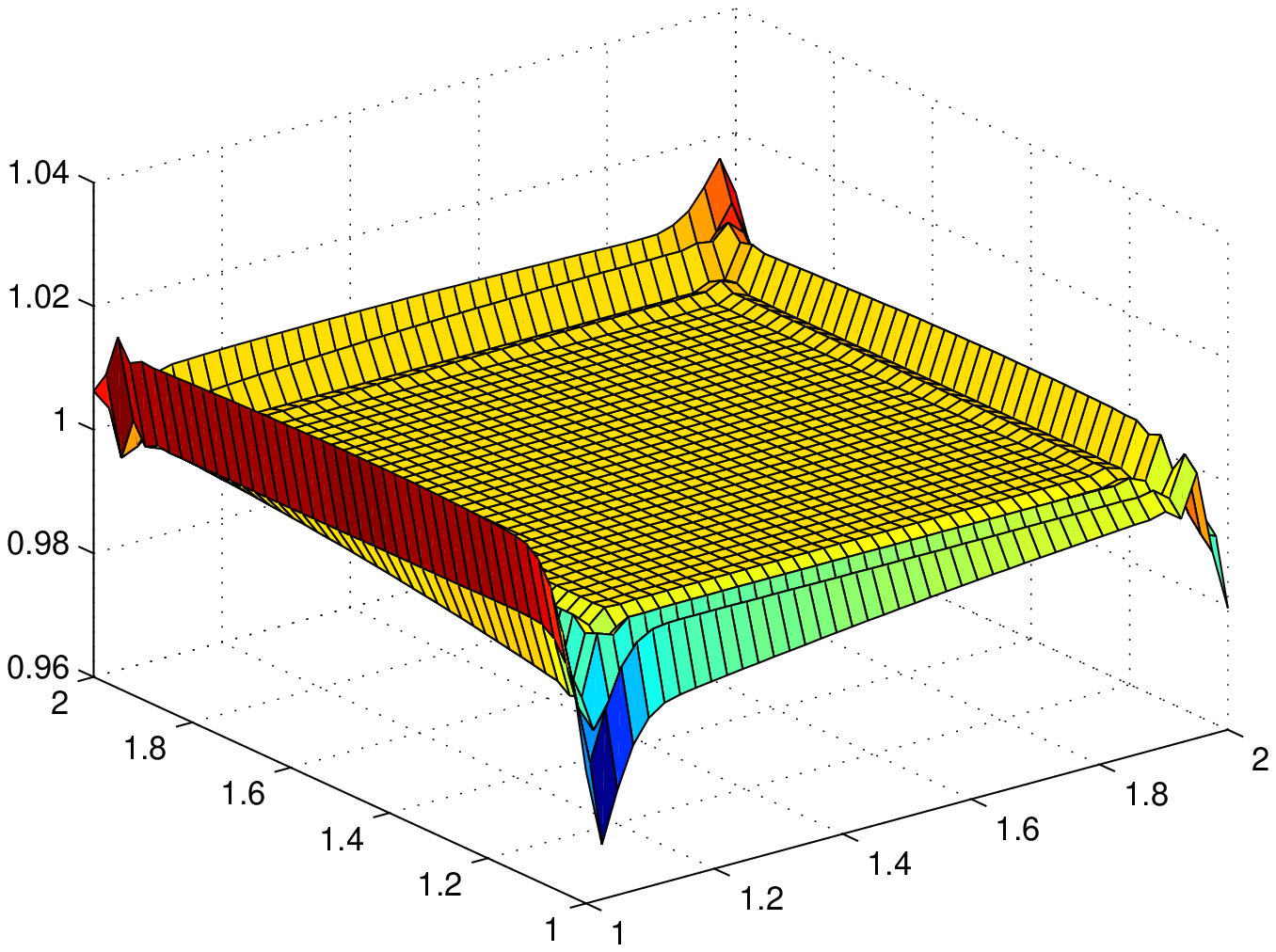}}
\end{minipage}\hfill%
\begin{minipage}[c]{0.45\textwidth}
\subfigure[$nu_y$ (AP-scheme) \label{uyapresvar}]{
\includegraphics[width=\textwidth]{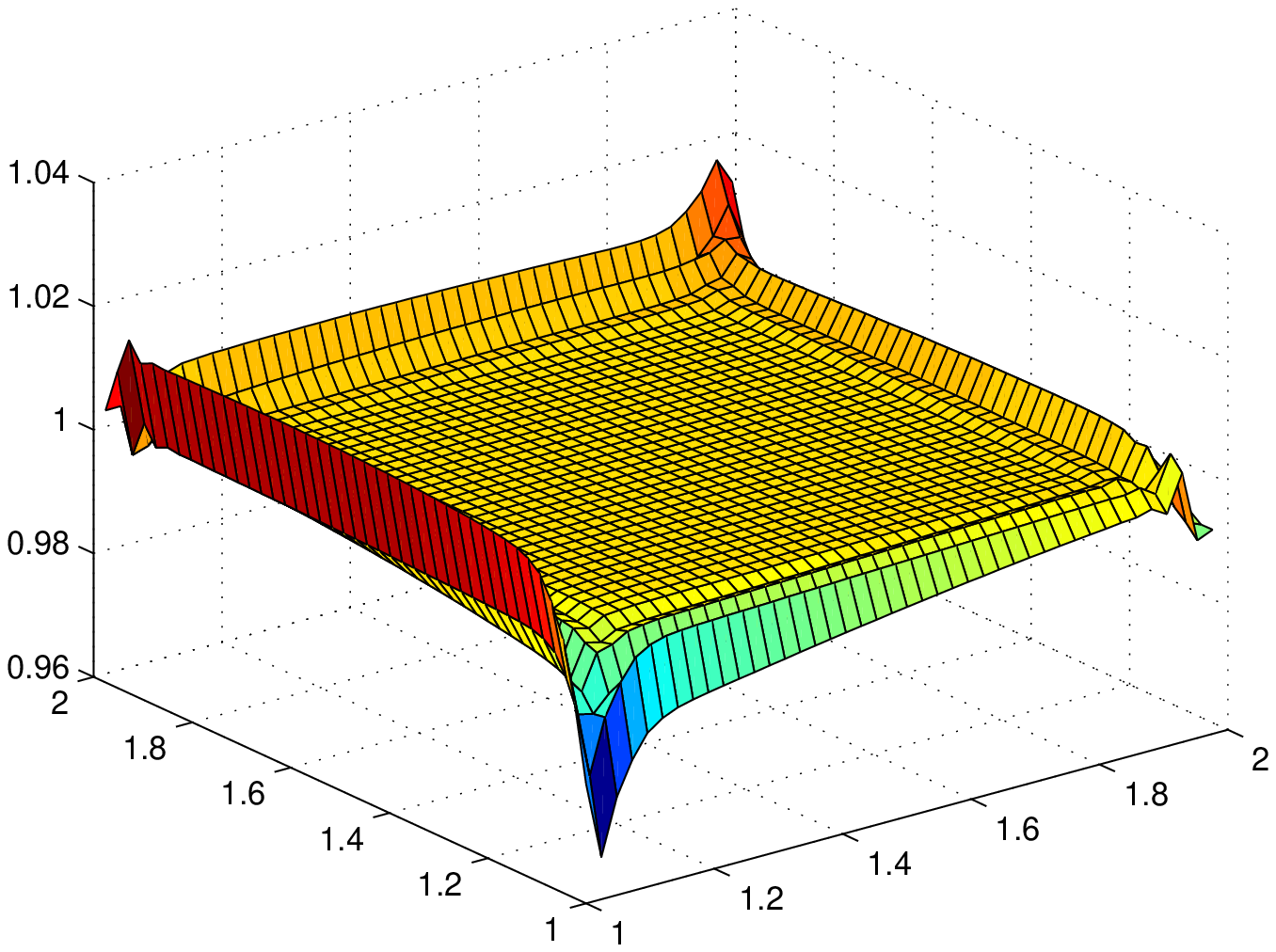}}
\end{minipage} 
\caption{Euler-Lorentz test case for a non uniform magnetic field in
  the resolved case at time $t=3.95 \; 10^{-6}$ s:
  density ($n$) and momentum ($nu_x$, $nu_y$) computed by the
  AP-scheme (right) and the classical scheme (left) for $\varepsilon =
  10^{-9}$ and $\Delta x = \Delta y = 1/40$.}
\end{center}
\end{figure}

\begin{figure}[!ht]
\begin{center}
\begin{minipage}[c]{0.45\textwidth}
\subfigure[$n$ (classical-scheme) \label{rhoclassnonresvar}]{
\includegraphics[width=\textwidth]{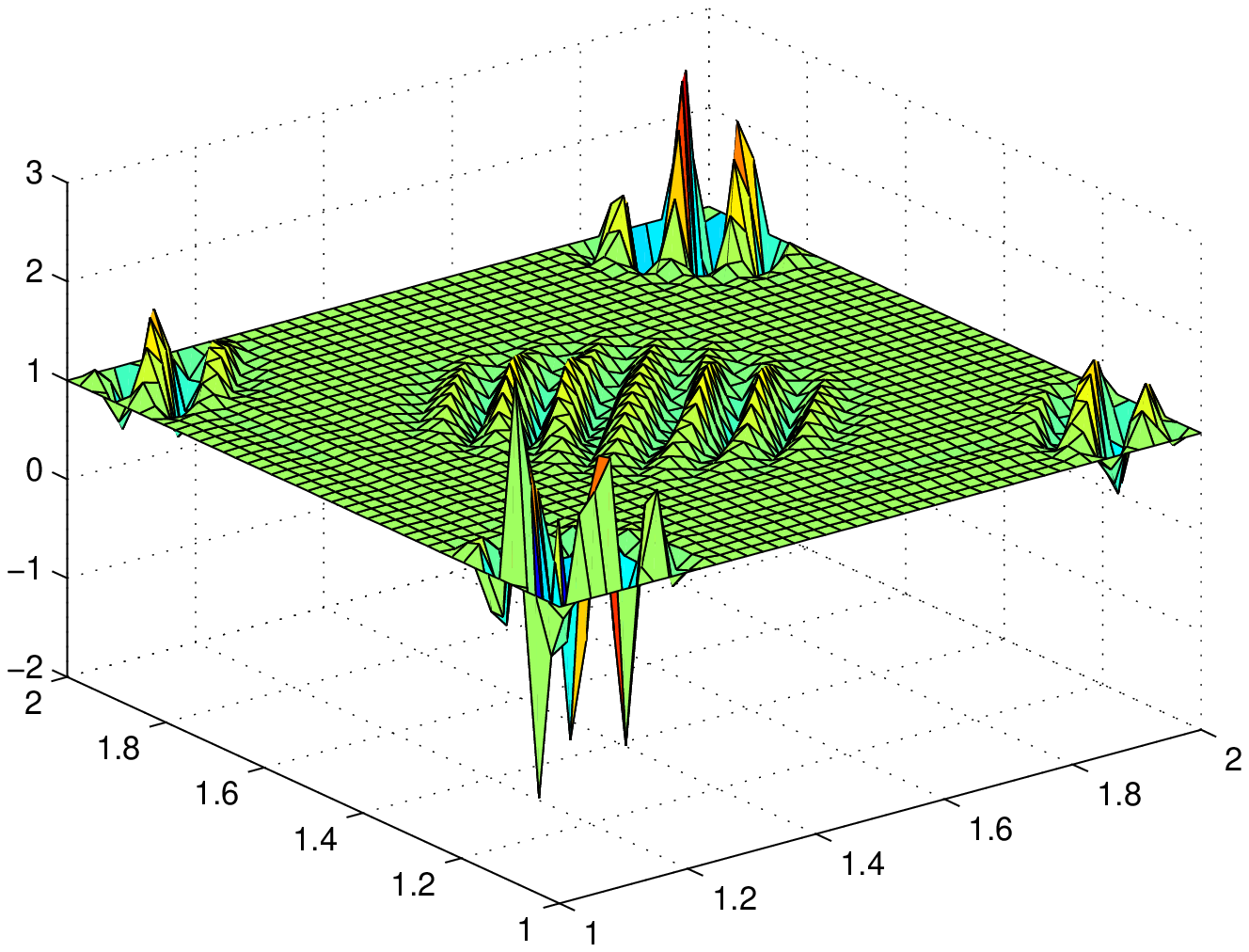}}
\end{minipage}\hfill%
\begin{minipage}[c]{0.45\textwidth}
\subfigure[$n$ (AP-scheme) \label{rhoapnonresvar}]{
\includegraphics[width=\textwidth]{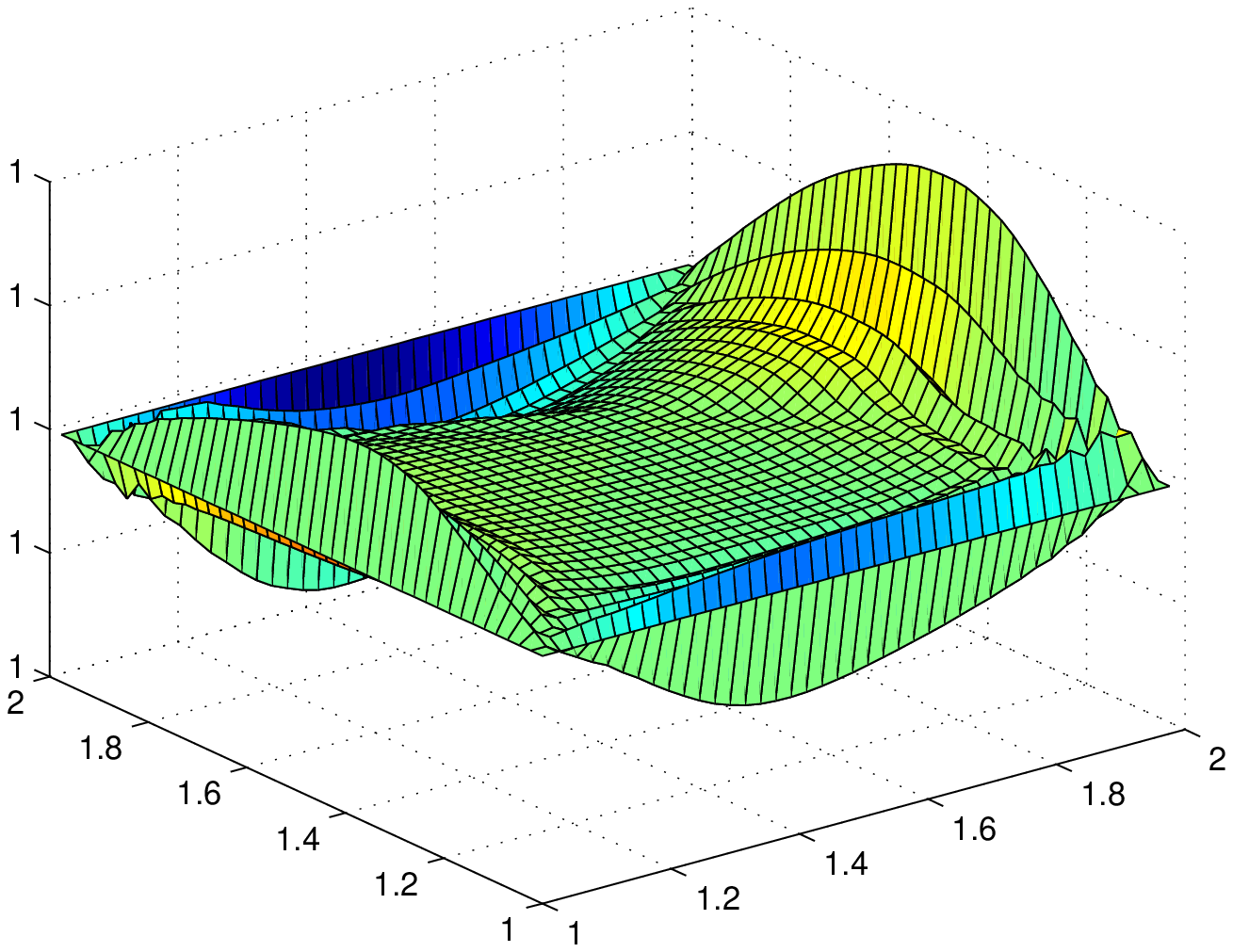}}
\end{minipage} 
\begin{minipage}[c]{0.45\textwidth}
\subfigure[$nu_x$ (classical-scheme) \label{uxclassnonresvar}]{
\includegraphics[width=\textwidth]{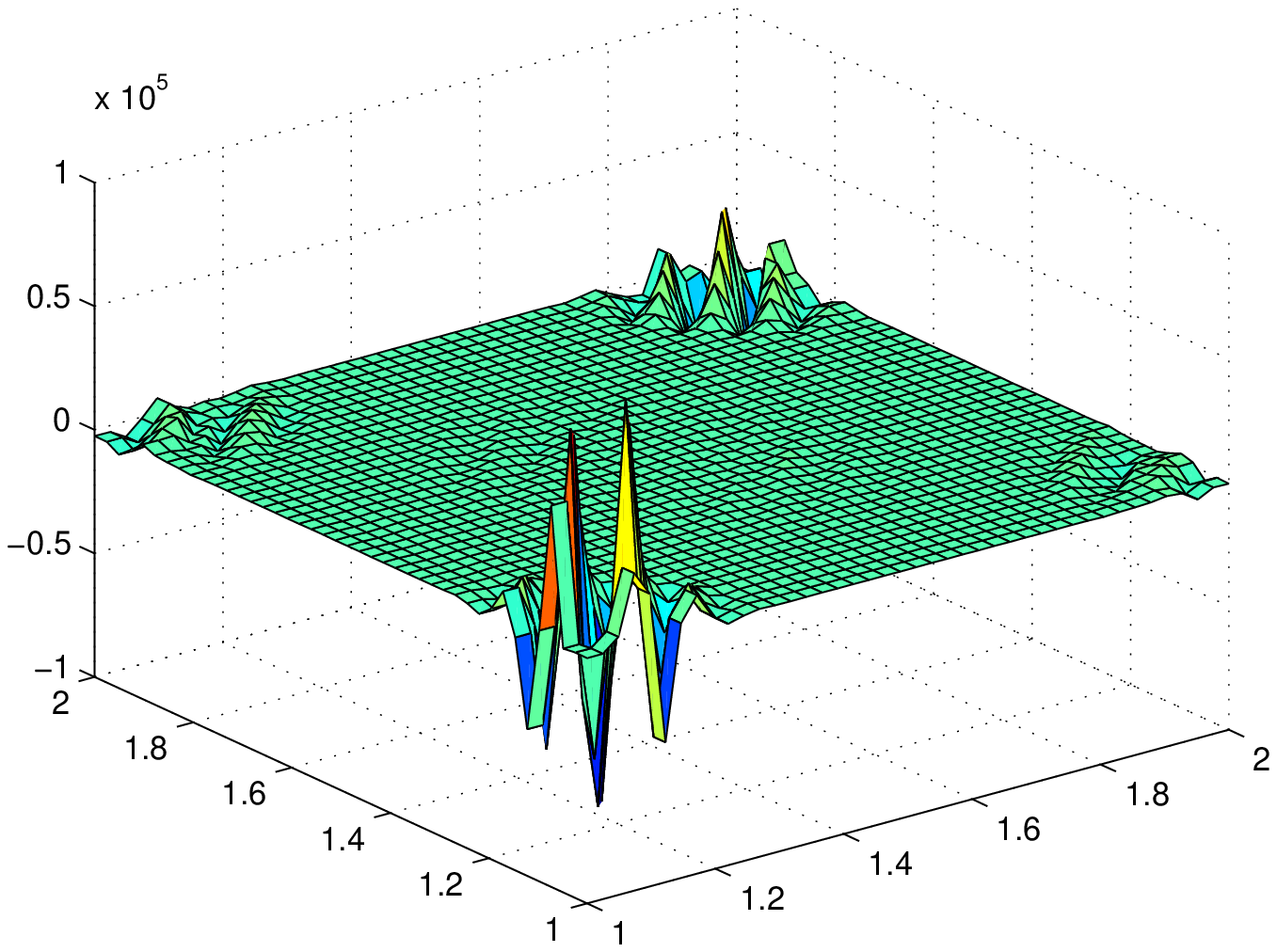}}
\end{minipage}\hfill%
\begin{minipage}[c]{0.45\textwidth}
\subfigure[$nu_x$ (AP-scheme) \label{uxapnonresvar}]{
\includegraphics[width=\textwidth]{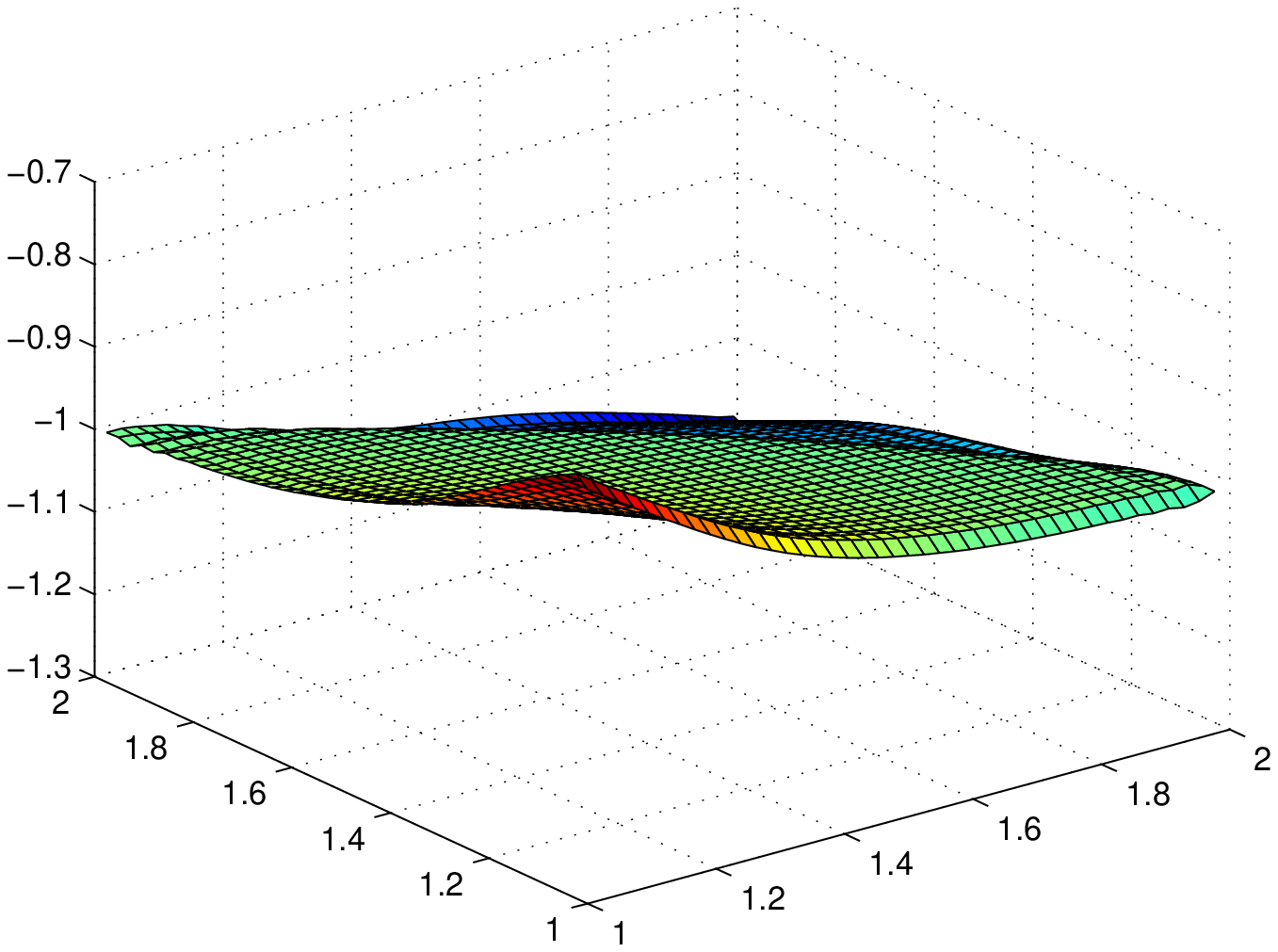}}
\end{minipage} 
\begin{minipage}[c]{0.45\textwidth}
\subfigure[$nu_y$ (classical-scheme) \label{uyclassnonresvar}]{
\includegraphics[width=\textwidth]{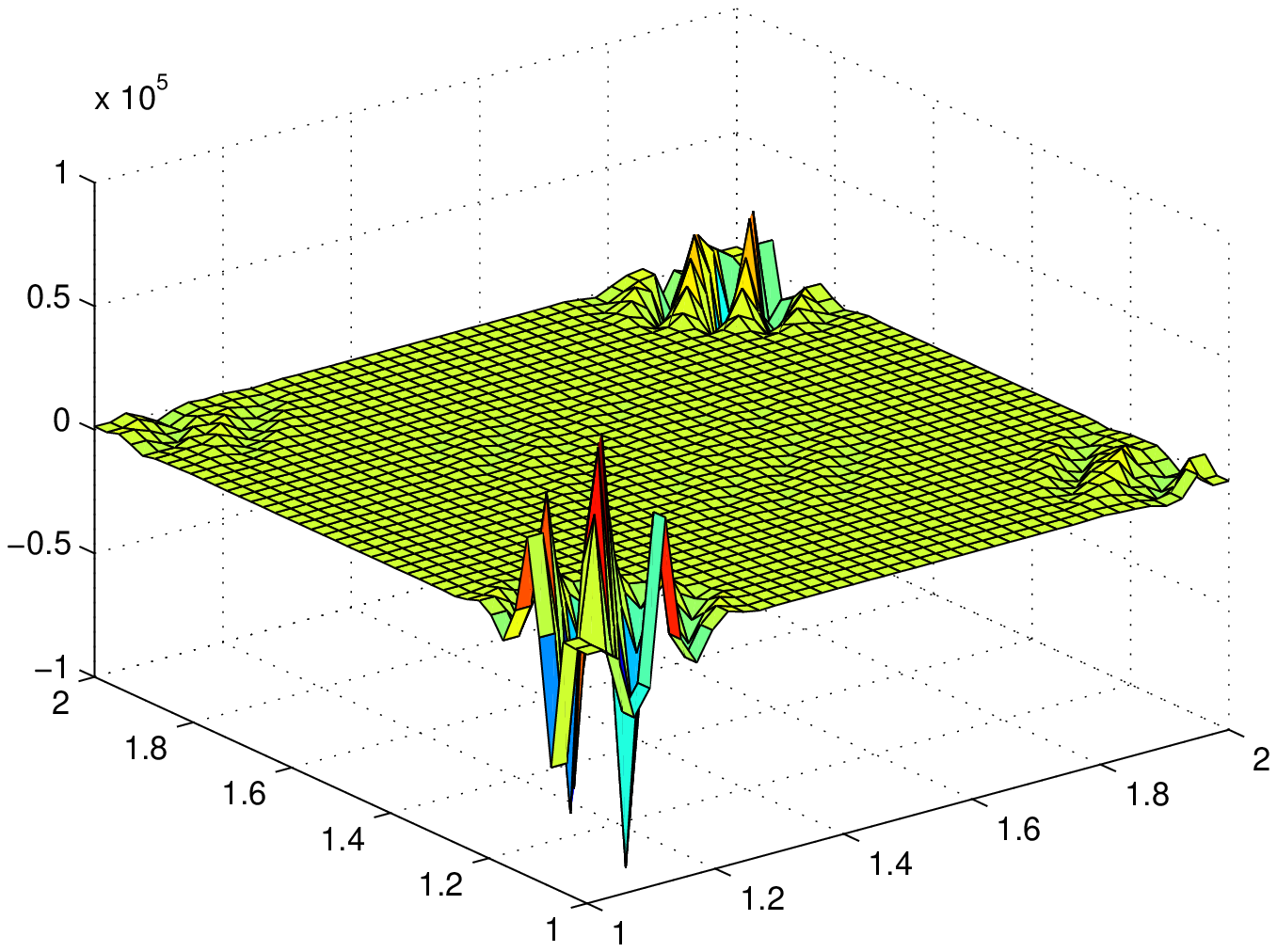}}
\end{minipage}\hfill%
\begin{minipage}[c]{0.45\textwidth}
\subfigure[$nu_y$ (AP-scheme) \label{uyapnonresvar}]{
\includegraphics[width=\textwidth]{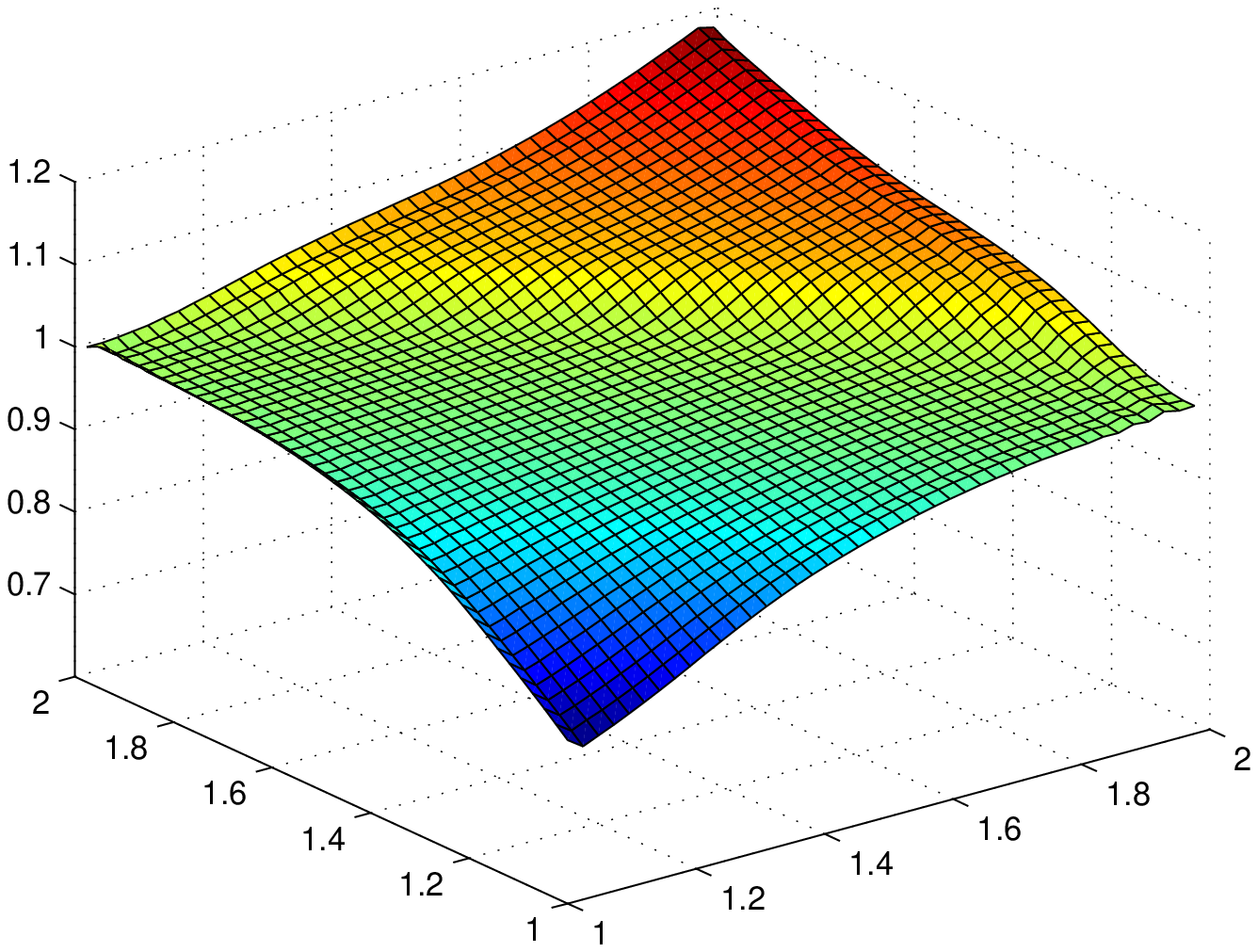}}
\end{minipage} 
\caption{Euler-Lorentz test case for a non uniform magnetic field in the
   under-resolved case at time $t=3.95 \; 10^{-5}s$:
  density ($n$) and momentum($nu_x$, $nu_y$) computed by the
  AP-scheme (right) and the classical scheme (left) for $\varepsilon =
  10^{-9}$ and $\Delta x = \Delta y = 1/40$.}
\end{center}
\end{figure}
Moreover as the initial conditions of the present test case are not stationary
solutions of the Euler-Lorentz model, it is important to check if the results
obtained in the non resolved case by the AP scheme correspond to the
proper limit regime. So we compare the results
obtained with and without the local perturbation on the initial
conditions. The difference between the results obtained with the two
simulations remain of the same order as the perturbation of the
initial condition. Fig. (\ref{diffrho}, \ref{diffv1},
  \ref{diffv2}) present the difference between the solutions obtained 
  with the perturbed and non-perturbed initial condition, for $n$, $(nu)_x$, $(nu)_y$. 
  The figures show that this difference is actually of $10^{-10}$ for the density and 
  $10^{-6}$ for the momenta. The difference with the value of $\varepsilon = 10^{-9}$, can be explained by the accumulation of the truncation error over the simulation time.

\begin{figure}[!ht]
\begin{center}
\begin{minipage}[c]{0.45\textwidth}
\subfigure[Difference for $n$ \label{diffrho}]{
\includegraphics[width=\textwidth]{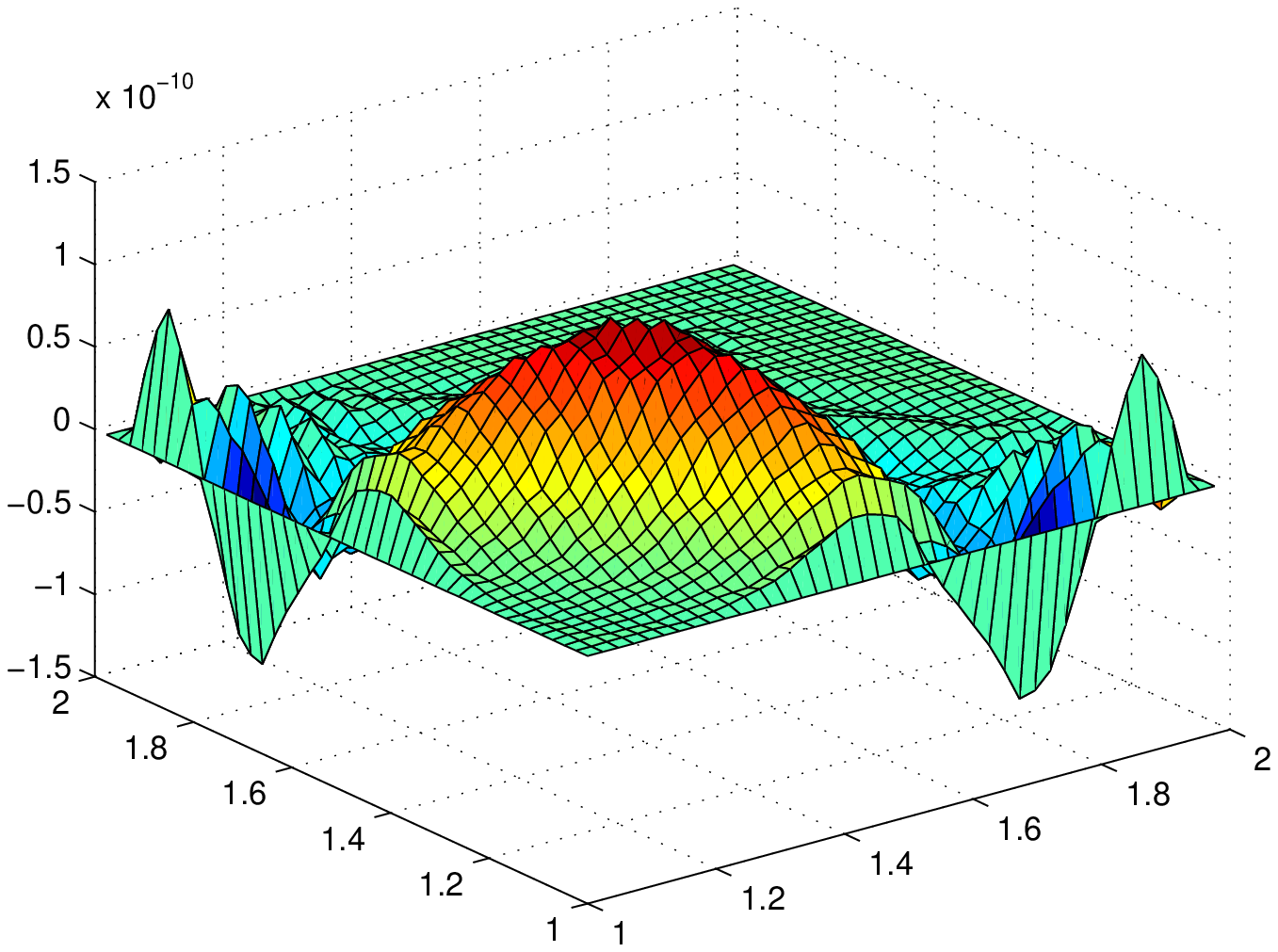}}
\end{minipage}\hfill%
\begin{minipage}[c]{0.45\textwidth}
\subfigure[Difference for $nu_x$ \label{diffv1}]{
\includegraphics[width=\textwidth]{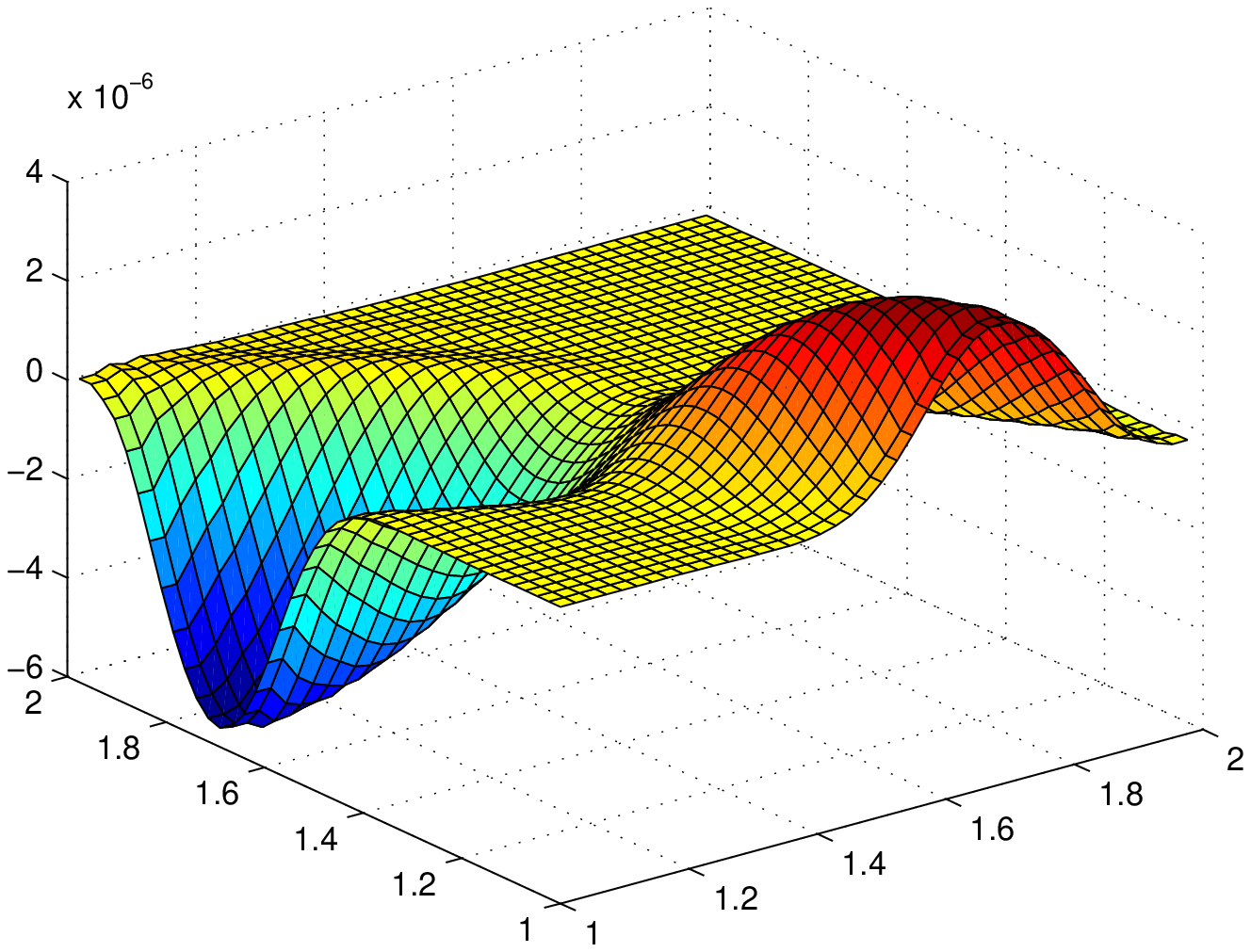}}
\end{minipage} 
\begin{minipage}[c]{0.45\textwidth}
\subfigure[Difference for $nu_y$ \label{diffv2}]{
\includegraphics[width=\textwidth]{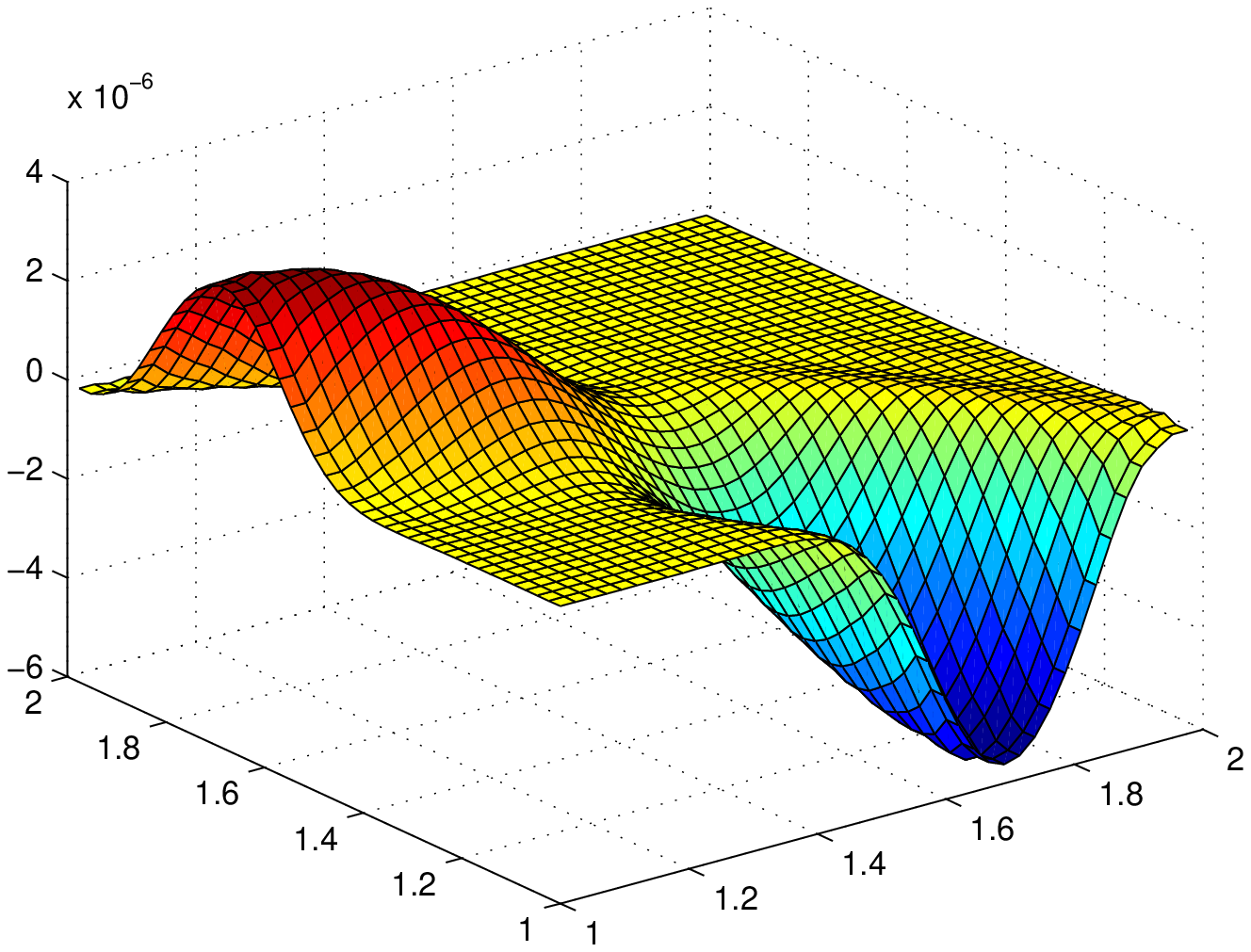}} 
\end{minipage} \hfill
\begin{minipage}[c]{0.5\textwidth}
\caption{Difference between the solutions obtained with an initial
  perturbation of order $\varepsilon=10^{-9}$ and the solution without
  any perturbation for the variable
  magnetic field after $1.58 \, s$ of simulation in the non resolved case.}\label{diff}
\end{minipage}
\end{center}
\end{figure}

\setcounter{equation}{0}
\section{Conclusion and perspectives}

A numerical method for degenerate anisotropic elliptic problems has been investigated. This method is based on a variational formulation together with a
decomposition of the solution. This problem has been
applied to the resolution of an Asymptotic-Preserving scheme for the isothermal Euler-Lorentz
system. Numerical simulations demonstrate the ability of the scheme to handle under-resolved situations where the time-step exceeds the CFL stability condition of the classical scheme. 

Forthcoming works will be devoted to the generalization of this
approach for the full Euler system with a non linear pressure law. In
this case non linear anisotropic elliptic problem have to be
handled. Moreover we can also deal with the more physical situation of
a plasma constituted by a mixture of ions and electrons. In this
situation the model can be described by the two-fluid Euler-Lorentz
system coupled with quasi-neutrality equation.

\bibliographystyle{abbrv}
\bibliography{bib}

\end{document}